\documentclass[12pt,reqno]{amsart}
\usepackage{amsmath,amsfonts,amssymb,amscd,amsthm,amsbsy,epsf}

\usepackage{ucs}
\usepackage[utf8]{inputenc}
\usepackage{mathabx}
\usepackage{caption}
\usepackage[UKenglish]{babel}
\usepackage{graphicx}
\graphicspath{{}}

\usepackage{hyperref}
\usepackage{comment}
\usepackage{multirow}
\usepackage{color}

\usepackage{longtable}

\title[KAM tori for the dissipative spin-orbit problem]{Efficient and
  accurate KAM tori construction for the dissipative spin-orbit
  problem using a map reduction}

\author[R. Calleja]{Renato Calleja}

\address{ Department of Mathematics and Mechanics, IIMAS, National
  Autonomous University of Mexico (UNAM), App. Postal 20-726,
  C.P. 0100, Mexico D.F. (Mexico)}

\email{celleja@mym.iimas.unam.mx}

\author[A. Celletti]{Alessandra Celletti}

\address{ Department of Mathematics, University of Rome Tor Vergata,
  Via della Ricerca Scientifica 1, 00133 Rome (Italy)}

\email{celletti@mat.uniroma2.it}

\author[J. Gimeno]{Joan Gimeno}
\address{ Department of Mathematics, University of Rome Tor Vergata,
  Via della Ricerca Scientifica 1, 00133 Rome (Italy)}

\email{gimeno@mat.uniroma2.it}

\author[R. de la Llave]{Rafael de la Llave}

\address{ School of Mathematics, Georgia Institute of Technology, 686
  Cherry St., Atlanta GA. 30332-0160 (USA) }

\email{rafael.delallave@math.gatech.edu}

\thanks{R.C. was partially supported by UNAM-DGAPA PAPIIT Project IN
  101020.  A.C. has been partially supported the MIUR Excellence
  Department Project awarded to the Department of Mathematics,
  University of Rome Tor Vergata, CUP E83C18000100006, EU H2020 MSCA
  ETN Stardust-Reloaded Grant Agreement 813644, and MIUR-PRIN
  20178CJA2B ``New Frontiers of Celestial Mechanics: theory and
  Applications''. J.G. has been supported by the Spanish grants
  PGC2018-100699-B-I00 (MCIU/AEI/FEDER, UE), the Catalan grant 2017
  SGR 1374 and MIUR-PRIN 20178CJA2B ``New Frontiers of Celestial
  Mechanics: theory and Applications''. The project leading to this
  application has also received funding from the European Union's
  Horizon 2020 research and innovation programme under the Marie
  Sk\l{}odowska-Curie grant agreement No 734557.  J.G. thanks the
  School of Mathematics of GT for its hospitality in Spring 2019 and
  Fall 2019. R.L has been supported by NSF grant DMS 1800241}

\date{\today}

\DeclareMathOperator{\Res}{Res}
\newcommand{\round}{\mathtt{round}}

\newcommand{\I}{\mathtt{i}}

\newcommand{\ve}[1]{\boldsymbol{#1}}
\newcommand{\eps}{\varepsilon}
\newcommand{\dis}{\eta}
\newcommand{\ecc}{e}
\newcommand{\omg}{\omega}

\newtheorem{thm}{Theorem}[section]
\newtheorem{meta-thm}[thm]{Meta-Theorem}
\newtheorem{cor}[thm]{Corollary}
\newtheorem{rem}[thm]{Remark}
\newtheorem{rems}[thm]{Remarks}
\newtheorem{defn}[thm]{Definition}

\newtheorem{lem}[thm]{Lemma}
\newtheorem{alg}[thm]{Algorithm}

\newcommand\beq[1]{ \begin{equation}\label{#1} }
\newcommand{\eeq}{ \end{equation} }

\newcommand\beqa[1]{ \begin{eqnarray} \label{#1}}
\newcommand{\eeqa}{ \end{eqnarray} }
\newcommand{\beqano}{ \begin{eqnarray*} }
\newcommand{\eeqano}{ \end{eqnarray*} }
\newcommand\equ[1]{{\rm (\ref{#1})}}
\newcommand{\R}{\mathbb{R}}
\newcommand{\Z}{\mathbb{Z}}

\newcommand{\T}{\mathbb{T}}
\newcommand{\M}{\mathcal{M}}

\begin{document}

\begin{abstract}
We consider the dissipative spin-orbit problem in Celestial Mechanics,
which describes the rotational motion of a triaxial satellite moving
on a Keplerian orbit subject to tidal forcing and \emph{drift}.

Our goal is to construct quasi-periodic solutions with fixed
frequency, satisfying appropriate conditions.

With the goal of applying rigorous KAM theory, we compute such
quasi-periodic solution with very high precision. To this end, we have
developed a very efficient algorithm.  The first step is to compute
very accurately the return map to a surface of section (using a high
order Taylor's method with extended precision).  Then, we find an
invariant curve for the return map using recent algorithms that take
advantage of the geometric features of the problem.  This method is
based on a rapidly convergent Newton's method which is guaranteed to
converge if the initial error is small enough. So, it is very suitable
for a continuation algorithm.

The resulting algorithm is quite efficient. We only need to deal with
a one dimensional function. If this function is discretized in $N$
points, the algorithm requires $O(N \log N) $ operations and $O(N) $
storage. The most costly step (the numerical integration of the
equation along a turn) is trivial to parallelize.

The main goal of the paper is to present the algorithms,
implementation details and several sample results of runs.

We also present both a rigorous and a numerical comparison of the
results of averaged and not averaged models.
\end{abstract}


\maketitle

\section{Introduction}
The construction of invariant structures in Celestial Mechanics and
Astrodynamics has become of great importance in recent times, both for
theoretical reasons and for the practical design of space missions.
At present, many space missions are based on the determination of
periodic and quasi-periodic orbits. Some notable examples of
periodic/quasi-periodic orbits used in mission design are Lyapunov,
Lissajous and halo orbits (see, e.g.,
\cite{CellettiPS2015,GomezM2001,JorbaM1999}).

The existence and persistence of quasi-periodic orbits was developed
by KAM theory (\cite{Kolmogorov54,Arnold63a,Moser62}).  Making KAM
theory into a practical tool is an ongoing and rapidly progressing
area, which applies to models of increasing complexity; it also leads
to applications and has uncovered new mathematical phenomena.

With these motivations, this work develops a method for the
construction of invariant tori in a concrete model of interest in
Celestial Mechanics, namely the dissipative spin-orbit problem (see
Section~\ref{sec:model}), which describes the rotation about its
center of an oblate moon orbiting a planet and subject to tidal
forces. Our goal is to develop a method to compute quasi-periodic
solutions in the spin-orbit problem. As it is well known,
quasi-periodic orbits can be described geometrically as invariant tori
on which the motion is conjugate to a rigid rotation. Hence, we will
use indistinctly the names quasi-periodic solution and invariant
(rotational) torus.

\subsection{Overview of the method}

The method we develop starts by constructing a surface of section and
a return map to it. The invariant tori for the flow correspond to
invariant tori for the return map.  We show that these return maps for
the spin-orbit have the remarkable property that they transform the
symplectic form into a multiple of itself.  These maps are called
\emph{conformally symplectic} systems and enjoy several remarkable
properties that lie at the root of a KAM theory and efficient
algorithms (see Section~\ref{sec:CS}).

To find the invariant torus of the map, we follow the approach in
\cite{CallejaCL11} and formulate a functional equation for the drift
parameter and for the embedding of the torus, whose solutions are
obtained formulating a quasi-Newton method that, given an approximate
solution of the functional equation, produces another one with a
quadratically small reminder. The quasi-Newton method in
\cite{CallejaCL11} takes advantage of the conformally symplectic
geometric property.

The results of \cite{CallejaCL11} guarantee that the method converges
(as a double exponential, as Newton's methods) if the initial error is
small enough (compared to some readily computable condition numbers).

The theorem in \cite{CallejaCL11} also shows that the difference
between the true solution and the initial approximation is controlled
by the error of the invariance equation, see
eq. \eqref{invariance}. Results of this form are called
\emph{a-posteriori theorems} in numerical analysis.  We also note that
one of the consequences of the work in \cite{CallejaCL11} is a local
uniqueness for the solutions of \eqref{invariance}, except for
composition for a rotation.  This lack of uniqueness comes from the
freedom on the choice of coordinates in the parameterization, but the
geometric object and the drift parameter are locally unique.

Since the iterative method converges for small enough error, it will
be used as the basis of a continuation method in parameters, which is
guaranteed to converge until the assumptions in the theorem fail.
Indeed, arguments in \cite{CallejaL10} show that the method can be
used as a practical way to compute the breakdown of the torus (see
\cite{CallejaC10} for an implementation to conformally symplectic
maps). Note that the method is backed up by theorems and guaranteed to
reach to the boundary of validity of the theorem, if given enough
computational resources.

In this paper we will implement the method with extended arithmetic
precision, motivated by the fact that the size of the error needed to
apply the a-posteriori theorem is typically smaller than what can be
obtained in double precision. With modern programming techniques,
writing extended precision programs is not much more time consuming
than using standard arithmetic.  Of course, there is a penalty in
speed, but since the algorithm is so efficient, one can still run
comfortably even with extended precisions in today's desktop machines
(see Section~\ref{sec:parallel} for details on timings and resources).
Of course, in continuations, it is also possible to run the first
iterations in double precision till the error is dominated by the
double precision round-off and then run the final iterations in
extended precision.

We have also used jet transport in order to get automatically the
(first order) variational flow with respect to initial conditions and
parameters. In \cite{CCGL20b} we use the jet transport to get high
order variational flows following the results in \cite{Gimeno2021}.

Finally, we have taken advantage of some modern advances such as
multicore machines and multithreading given, for instance,
continuation iterations in around 1 min when the initial guess is
small enough. We did not explore other advanced architectures such as
GPU, whose application in Celestial Mechanics is an interesting
challenge. Some work on a simpler problem is in \cite{KumarAL21c}.

\subsection{Efficiency and accuracy of the method}

The use of return maps is very economical and natural.  In the study
of invariant tori for differential equations, it is standard to
separate the directions along the flow and the transversal
directions. 

On the one hand, the torus remains very smooth along the directions
of the flow for all values of the perturbation parameter. The flow
in these directions can just be reduced to the well studied problem of
computing solutions of ODE's and, in fact, there are many different
algorithms suitable in different conditions.

On the other hand, the computation of the torus in the directions of
the section is a much more complicated problem, since it requires KAM
theory and the tori along these directions are much less regular;
indeed, for large enough values of the perturbation, the tori may
disappear.  Nevertheless, even for values of the perturbation
parameter where the tori do not exist, the solutions of the
differential equations can be comfortably computed.

By dividing the problem into the KAM part for maps and the propagation
to the return section, we reduce significantly the difficulty of the
KAM part, since the tori have lower dimension.  The computation of the
return map is more complicated, but it is easily parallelizable and
there are many studies on optimizing it.  Hence the break up ends
being rather advantageous. As indicated above, the methodologies of
the two parts are very different and each of them can be fine tuned
separately.

The KAM part for maps is well documented in \cite{CallejaCL11}; the
algorithm therein applies quasi-Newton corrections and takes advantage
of several identities related to the fact that the map is conformally
symplectic.

The number of explicit steps of the KAM iterative procedure is about a
dozen, see Algorithm~\ref{alg.newton}.  All the elementary steps are
well structured vector operations that are primitives in modern
languages or libraries, so that they are not too cumbersome to
program.

Quite remarkably, all the steps are diagonal either in a grid
representation of the function or in a Fourier representation. Of
course, we can switch from one representation to the other using
FFT. Hence, for a function discretized in $N$ modes, the quadratically
convergent method requires only $O(N)$ storage and $O(N \log N)$
operations.  Note that, in modern computers, the vector operations and
the FFT are highly optimized, with specialized hardware.

These KAM algorithms have been implemented for maps given by explicit
simple formulas \cite{CallejaC10, CallejaF2012, CallejaCL20}. In
theory, the only thing that one would need to do is to use the return
map of the ODE (and its variational equations) in place of the
explicit formulas.  However, as we report in \cite{CCGL20c}, in
contrast with the explicit formulas that have few important harmonics,
the return maps have many more relevant harmonics; this requires some
adaptations and the phenomena observed are different.

\subsubsection{Relation with other methods}

The methods of computing invariant tori based on normal form theory
require working with functions with as many variables as the phase
space, see \cite{StefanelliL12,StefanelliL15}. In contrast, our
methods require to manipulate only functions with as many variables as
the dimension of the tori of the map. Reducing the number of variables
in the unknown function is very important, since the number of
operations needed to manipulate a function grows exponentially (with a
large exponent) with the number of variables.  Some recent papers that
are also using return maps in Celestial Mechanics are, for instance,
\cite{HaroM21} (full dimensional tori in Hamiltonian systems) and
\cite{KumarAL21b} (whiskered tori, their stable and unstable manifolds
and their intersections in Hamiltonian systems).

It is interesting to compare the methods developed here to
\cite{Olikara16}, which uses a discretization of the tori without
separating the tori and the flow directions.  If the torus is
discretized in $N$ points, this method requires $O(N^2)$ storage and
$O(N^3)$ operations. Notice that $N$ points in 2-D tori give more or
less the same precision as $N^{\frac{1}{2}}$ in 1-D tori.

A curious remark is that the linearization of the invariance equation
\eqref{invariance} has a spectrum lying in circles as shown by
\cite{Mather68} (see \cite{AdomaitisKL07, HaroL2} for numerical
experimentations), so that the Arnold-Kyrlov methods, successful in
other models (\cite{Simo2010}), do not work very well. The method we
use can be understdood as saying that, using geometric identities, we
get the linearized equations to become constants (this phenomenon is
called \emph{automatic reducibility}).

One improvement in continuation methods, after an expensive effort in
one step, is that one can compute inverses (\cite{Moser73,Hald75}) or
diagonalizations (\cite{HLlnum, JorbaO09}), perturbatively. These
methods require still to store $O(N^2)$ storage and the perturbative
calculations still require $O(N^3)$ operations even if the constants
improve.

\subsection{The model}

We have implemented our results to the so-called dissipative
spin-orbit model (\cite{Celletti2010}). In this section, we will
review the physical bases of the model as well as formulate several
variants (so-called, time-dependent friction \eqref{eq.spinxy} and
averaged friction \eqref{eq.avg}).  These models will be analyzed
(numerically and rigorously) in subsquent sections.

The spin-orbit model describes the rotational motion of a triaxial
non-rigid satellite whose center of mass moves along an elliptic
Keplerian orbit around a central planet.  The spin-axis of the
satellite is assumed to be perpendicular to the orbital plane and
coinciding with the shortest physical axis. The rotation angle of the
satellite is the angle between the longest axis of the satellite and a
fixed direction, e.g. the periapsis line.

The motion of the rotation angle satisfies a second order differential
equation depending periodically on time, through the orbital elements
describing the osculating position of the center of mass of the
satellite; such equation depends on two parameters, namely the orbital
eccentricity and the equatorial flattening of the satellite.  The
model equations include a dissipative term due to the non-rigidity of
the satellite, since the rotation gives rise to tides that dissipate
energy and generate a torque.  We adopt the model of \cite{Peale} in
which the tidal torque is proportional to the angular velocity with a
time-periodic coefficient. The dissipative term depends on two
parameters: the orbital eccentricity and the dissipative factor, which
is determined by the physical features of the satellite.

For typical bodies of the solar system, e.g. the Moon and many others
among the biggest satellites, the force induced by the dissipation is
much smaller than the conservative part, so that dissipation can be
ignored in the description over short times.  Nevertheless, since
the dissipative forces have consequences that accumulate over time,
they are very important in the description of long-term effects.

In the applied literature, it is very common to use a simplified
version of the tidal torque that can be obtained by averaging the
dissipation over time (\cite{ARMA,CellettiL2014,LaskarC}), so that the
tidal torque becomes proportional to the derivative of the rotation
angle.

One of the advantages of the method in this paper is that it provides
a rather general rigorous justification of the averaging method for
quasi-periodic solutions The basic idea is very simple: using standard
averaging methods, we control the $2\pi$ time map and, then, the
effect of changing the map is controlled by the a-posteriori theorem
(see Appendix~\ref{sec:averaging}). For the spin-orbit problem we also
provide another justification that applies to all solutions.

Our formalism provides also rigorous estimates on the validity of the
averaging approximation. Standard averaging theory can give estimates
on the difference between the return maps in the averaged model and
the true model (notice that the time of flights in return maps is
about $1$, so that controlling the averaging method to order $1$ is
very standard). Then, we can use the a-posteriori format of the KAM
theory in \cite{CallejaCL11} to obtain estimates on the difference
between the KAM tori and the drift parameters in the averaged and
non-averaged cases. This result may look surprising, since one obtains
control on solutions (and drift) for very long times. Besides this
very general perturbative argument, in Appendix~\ref{sec:averaging} we
present some elementary arguments that, taking advantage of the
structure of the system, obtain non-perturbative results.

We conclude by mentioning that the current work has several
consequences: in \cite{CCGL20b} we study the quantitative verification
of a-posteriori theorems and the quantitative condition numbers, while
in \cite{CCGL20c} we explore numerically the boundary of validity of
KAM theory and uncover several phenomena that deserve further
mathematical investigation. We hope that this paper (and the
companions \cite{CCGL20b} and \cite{CCGL20c}) can stimulate further
research, for example turning the estimates in \cite{CCGL20b} into
rigorous computer assisted proofs, studying higher dimensional models,
incorporating more advanced computer architectures and explaining the
phenomena at breakdown.

\subsection{Organization of this paper}

This work is organized as follows.  In Section~\ref{sec:model} we
present the spin-orbit model with tidal torque.  The numerical
formulation of the spin-orbit problem is given in
Section~\ref{sec:numerical}, while the spin-orbit map is derived in
Section~\ref{sec:spin-map}. The algorithm for the construction of
invariant attractors and its applications is presented in
Section~\ref{sec:attractors}. Finally, some conclusions are given in
Section~\ref{sec:conclusions}.

\section{The spin-orbit problem with tidal torque}\label{sec:model}

Consider the motion of a rigid body, say a satellite $\mathcal{S}$,
with a triaxial structure, rotating around an internal spin-axis and,
at the same time, orbiting under the gravitational influence of a
point-mass perturber, say a planet $\mathcal{P}$.  A simple model that
describes the coupling between the rotation and the revolution of the
satellite goes under the name of \sl spin-orbit problem, \rm which has
been extensively studied in the literature in different contexts (see,
e.g., \cite{Beletsky,Celletti1990,Celletti90II,LaskarC,Wisdom}).  This
model is based on some assumptions that we are going to formulate as
follows. Let $\mathcal{A} < \mathcal{B} < \mathcal{C}$ denote the
principal moments of inertia of the satellite $\mathcal{S}$; then, we
assume that:
\begin{enumerate}
\renewcommand*{\theenumi}{\bf H\arabic{enumi}}
\renewcommand*{\labelenumi}{\theenumi.}
\setlength{\itemsep}{.5em}
 \item \label{assump1} The satellite $\mathcal{S}$ moves on an
   elliptic Keplerian orbit with semimajor axis $a$ and eccentricity
   $\ecc$, and with the planet $\mathcal{P}$ in one focus;
 \item \label{assump2} The spin-axis of the satellite coincides with
   the smallest physical axis of the ellipsoid, namely the axis with
   associated moment of inertia $\mathcal{C}$;
 \item \label{assump3} The spin-axis is assumed to be perpendicular to
   the orbital plane;
 \item \label{assump4} The satellite $\mathcal{S}$ is affected by a
   tidal torque, since it is assumed to be non-rigid.
\end{enumerate}
We adopt the units of measure of time such that the orbital period,
say $T_{orb}$, is equal to $2\pi$, which implies that the mean motion
$n=2\pi/T_{orb}$ is equal to one.

We define the equatorial ellipticity as the parameter $\eps>0$ given
by
\begin{equation}
 \label{eq.eps}
 \eps = \frac{3}{2} \frac{\mathcal{B}-\mathcal{A}}{\mathcal{C}}\ ,
\end{equation}
which is a measure of the oblateness of the satellite.  When $\eps =
0$, then $\mathcal{A}=\mathcal{B}$ which means that the satellite is
symmetric in the equatorial plane and, because of \eqref{assump3}, it
coincides with the orbital plane.

\vskip.1in

We consider the perturber $\mathcal{P}$ at the origin of an inertial
reference frame with the horizontal axis coinciding with the direction
of the semimajor axis.  It is convenient to identify the orbital plane
with the complex plane $\mathbb{C}$. The location of the center of
mass of the rigid body $\mathcal{S}$ with respect to the perturber
$\mathcal{P}$ is given, in exponential form, by $r \exp(\I f) \in
\mathbb{C}$, where $r > 0$ and $f$ are real functions depending on the
time $t$ and they represent respectively the \emph{instantaneous
  orbital radius} and the \emph{true anomaly} of the Keplerian orbits.
Indeed, over time, $r$ and $f$ describe an ellipse of eccentricity
$\ecc \in [0,1)$, semimajor axis $a$ and focus at the origin, see
  Figure~\ref{fig.spin}.

\vskip.1in

Given that the mean motion $n$ has been normalized to one, then the
\emph{mean anomaly} coincides with the time $t$. By Kepler's equation
(\cite{Celletti2010}), we have the following relation between the
\emph{eccentric anomaly} $u$ and the time:
\begin{equation}
 \label{eq.tu}
 t = u - \ecc \sin u\ .
\end{equation}
The expressions which relate $r$ and $f$ with the eccentric anomaly
(and hence with time through \eqref{eq.tu}) are given by
\begin{align}
 \label{eq.r}
 r &= a(1 - \ecc \cos u)\ , \\
 \label{eq.f}
 r \exp(\I f) &= a(\cos u - \ecc + \I \sqrt{1 - \ecc ^2} \sin u)\ .
\end{align}
Notice that we are assuming in \eqref{eq.f} that for $t=0$, $f(0) =
u(0) = 0$, and consequently, $f(\pi) = u(\pi) = \pi $ when $t = \pi$.
We also recall the following relations between the Keplerian elements,
that will be useful in the following:
\begin{equation}
 \label{eq.cfsf}
 \cos f = \frac{\cos u - \ecc}{1 - \ecc \cos u} \qquad \text{and}
 \qquad \sin f = \frac{\sqrt{1 - \ecc ^2} \sin u}{1 - \ecc \cos u}\ .
\end{equation}

As for the rotational motion, let $x$ be the angle formed by the
direction of the largest physical axis, which belongs to the orbital
plane, due to the assumptions \eqref{assump2} and \eqref{assump3},
with the horizontal (or semimajor) axis $a$.\\

If we neglect dissipative forces, the equation of motion which gives
the dependence of $x$ on time is given by the following expression
(\cite{Celletti2010}) to which we refer as the \sl conservative
spin-orbit equation: \rm
\begin{equation}
 \label{eq.spin-cons}
 \frac{d^2x(t)}{dt^2} + \eps \biggl(\frac{a}{r(t)}\biggr)^3 \sin
 \bigl(2 x(t) - 2f(t)\bigr) = 0\ ,
\end{equation}
where $\eps > 0$ is given in \eqref{eq.eps}, $r(t)=r(u(t);\ecc)$ in
\eqref{eq.r}, $f(t)=f(u(t);\ecc)$ in \eqref{eq.cfsf}, and where $u$ is
related to $t$ through \equ{eq.tu}.

\vskip.1in

\begin{figure}[ht]
 \includegraphics[scale=.8]{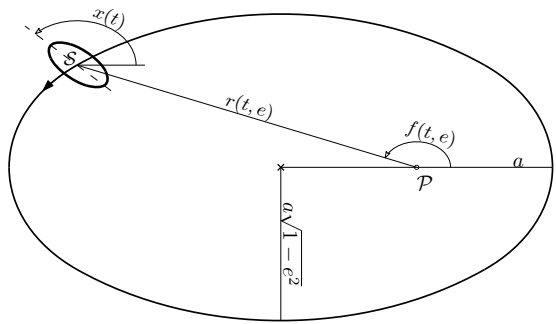}
 \caption{The spin-orbit problem: a triaxial satellite $\mathcal{S}$
   moves around a planet $\mathcal{P}$ on an elliptic orbit with
   semimajor axis $a$ and eccentricity $e$. The position of the
   barycenter of $\mathcal{S}$ is given by the orbital radius $r$ and
   the true anomaly $f$.  The rotational angle is denoted by $x$.}
 \label{fig.spin}
\end{figure}

If we now assume that the satellite is not rigid, then we must
consider a tidal torque, say $\mathcal{T}_d$, that acts on the
satellite. According to \cite{Macdonald,Peale}, we can write the tidal
torque as a linear function of the velocity:
\begin{equation}
 \label{eq.mcdonald}
 \mathcal{T} _d \biggl(\frac{dx(t)}{dt}, t\biggr) = - \dis
 \biggl(\frac{a}{r(t)}\biggr)^6 \biggl(\frac{dx(t)}{dt} - \frac{d f(t)}{dt}
 \biggr)\ ,
\end{equation}
where $\dis > 0$ is named the \sl dissipative constant. \rm Since we
are interested in astronomical applications, we specify that $\dis$
depends on the physical and orbital features of the body and takes the
form
\[
\dis = 3n\ \frac{k_2}{\xi Q} \biggl(\frac{R_e}{a}\biggr)^3
     \frac M m\ ,
\]
where $k_2$ is the second degree potential \sl Love number \rm
(depending on the structure of the body), $Q$ is the so--called \sl
quality factor \rm (which compares the frequency of oscillation of the
system to the rate of dissipation of energy), $\xi$ is a structure
constant such that $C=\xi mR_e^2$, $R_e$ is the equatorial radius, $M$
is the mass of the central body $\mathcal{P}$, $m$ is the mass of the
satellite $\mathcal{S}$. Astronomical observations suggest that for
bodies like the Moon or Mercury the dissipative constant $\dis$ is of
the order of $10^{-8}$.

\vskip.1in

The dynamics including the tidal torque is then described by the
following equation to which we refer as the \sl dissipative spin-orbit
equation: \rm
\begin{equation}
 \label{eq.spinxy}
 \frac{d^2x(t)}{dt^2} + \eps \biggl(\frac{a}{r(t)}\biggr)^3 \sin \bigl(2 x(t) -
 2f(t)\bigr) = - \dis \biggl(\frac{a}{r(t)}\biggr)^6 \biggl(\frac{dx(t)}{dt} -
 \frac{d f}{dt} \biggr)\ .
\end{equation}
The expression for the tidal torque can be simplified by assuming (as
in \cite{Peale,LaskarC}) that the dynamics is essentially ruled by the
average $\overline{\mathcal{T}}_d$ of the tidal torque over one
orbital period, which can be written as
\begin{equation}
 \label{eq.avg}
\overline{\mathcal{T}}_d \biggl(\frac{dx}{dt} \biggr) = -\dis
\biggl(\bar L(\ecc)\frac{dx}{dt} -\bar N(\ecc)\biggr) \ ,
\end{equation}
where (compare with \cite{Peale})
\begin{equation*}
 \begin{split}
\bar L(\ecc)&\equiv {1\over{(1-\ecc^2)^{9/2}}}
\biggl(1+3\ecc^2+{3\over 8}\ecc^4 \biggr) \ , \\ \bar N(\ecc)&\equiv
      \frac{1}{(1-\ecc^2)^6} \biggl(1+{{15}\over 2}\ecc^2+{{45}\over
        8}\ecc^4+ {5\over {16}}\ecc^6 \biggr)\ .
 \end{split}
\end{equation*}
When considering the averaged tidal torque, one is led to study the
following equation of motion to which we refer as the \sl averaged
dissipative spin--orbit equation: \rm
\begin{equation}
 \label{eq.avg-spinxy}
\frac{d^2x(t)}{dt^2} + \eps \biggl({a\over r(t)}\biggr)^3 \sin
\bigl(2x(t)-2f(t) \bigr) = -\dis \biggl(\bar
L(\ecc)\frac{dx(t)}{dt}-\bar N(\ecc)\biggr)\ .
\end{equation}

Note that in this model, we consider the average of the tidal forces,
but do not average the conservative forces (see
Appendix~\ref{sec:averaging}). This is justified because, as indicated
before, in practical problems, the dissipative forces are much smaller
than the conservative ones.

\begin{rems}
\begin{enumerate}
\renewcommand*{\theenumi}{\emph{\roman{enumi}}}
\renewcommand*{\labelenumi}{(\theenumi)}
 \item The parameter $\eps$ in \eqref{eq.eps} is zero only in the case
   of an equatorial symmetry with $\mathcal{A} = \mathcal{B}$. In that
   case, the equation of motion \eqref{eq.spin-cons} is trivially
   integrable.
 \item When $\ecc =0$, the orbit is circular and therefore $r = a$ and
   $f = t + t _0$.  Also in this case the equation of motion
   \eqref{eq.spin-cons} is integrable.
 \item The equation \eqref{eq.spin-cons} is associated to the
   following one-dimensional, time-dependent Hamiltonian function:
 \begin{equation}
 \label{eq.hamilto}
  \mathcal{H}(y,x,t) = \frac{y^2}{2} - \frac{\varepsilon}{2}
  \Big(\frac{a}{r(t)}\Big)^3 \cos(2x - 2f(t))\ .
 \end{equation}
 \item Equations \eqref{eq.spin-cons}, \eqref{eq.spinxy} and
   \eqref{eq.avg-spinxy} are defined in a phase space which is a
   subset of $[0,2\pi)\times \mathbb{R}$. Such a phase space can be
     endowed with the standard scalar product and a symplectic form
     $\Omega$, which in our case, is just the two dimensional area in
     phase space.  Even if in two dimensional phase spaces the area
     (volume) is the same as the symplectic manifold, in systems with
     $N$ degrees of freedom ($N > 1$) the preservation of the two-form
     $\Omega$ is much more stringent than the preservation of the $2N$
     dimensional volume.

\end{enumerate}
\end{rems}

\section{Numerical version of the spin-orbit problem}
\label{sec:numerical}

To get a numerical representation of the ordinary differential
equation \equ{eq.spinxy} (equivalently \equ{eq.spin-cons} or
\equ{eq.avg-spinxy}), it is convenient to express the equation in
terms either of the eccentric anomaly $u$ or the mean anomaly which
coincides with $t$. Although there is a clear bijection between $t$
and $u$ through \eqref{eq.tu}, it seems reasonable to redefine
everything in terms of the eccentric anomaly $u$ due to the
expressions of $r$ in \eqref{eq.r} and $f$ in \eqref{eq.f}. The
procedure to get the equation of motion with $u$ as independent
variable is the following.

The expression of $f$ in \eqref{eq.cfsf} is given easily in terms of
$u$ (and $\ecc$). Let $s(x)=s(x;u,\ecc)$ be the function defined as
$s(x;u,\ecc)\coloneq \sin(2x(t)-2f(t))$, where the dependence on $u$,
$\ecc$ enters through $f$.  Using trigonometric identities, we have an
explicit expression in terms of $u$ and $\ecc$ for the sinus in
\eqref{eq.spinxy}:
\begin{equation}
 \label{eq.sx}
s(x; u, \ecc) = \sin(2x) (2\cos^2 f-1) - \cos(2x)2\cos f \sin f\ .
\end{equation}
Note here the useful relation for the derivatives of $s,c$:
\begin{equation}
  \begin{split}
& \frac{\partial s}{\partial x}(x; u, \ecc) = 2c(x; u, \ecc) \ , \\
    &  \frac{\partial c}{\partial x}(x; u, \ecc) = -2 s (x; u, \ecc) \ ,
\end{split}
\end{equation}
where
\begin{equation}
 \label{eq.cx}
 c(x) = c(x; u, \ecc) \coloneq \cos(2x) (2\cos^2 f-1) + \sin(2x)2\cos f
 \sin f\ .
\end{equation}

The time-change given by \eqref{eq.tu} leads to
\begin{equation*}
 \frac{d f}{dt} = \left(\frac{a}{r}\right)^2 \sqrt{1 - \ecc ^2}\ .
\end{equation*}
As a consequence, equation \eqref{eq.spinxy} can be expressed in terms
of the independent variable $u$ as
\begin{equation}
 \label{eq.spinbg}
 \frac{d^2 \beta(u)}{du ^2} - \frac{d\beta(u)}{du} \frac{a}{r(u)} \ecc
 \sin u + \eps \frac{a}{r(u)} s(\beta) = - \dis \biggl(\frac{a}{r(u)}\biggr)^5
 \biggl(\frac{d\beta(u)}{du} - \frac{a}{r(u)} \sqrt{1 - \ecc ^2}\biggr)
\end{equation}
with $r$ defined in \eqref{eq.r} and $s(\beta) \coloneq s(\beta; u,
\ecc)$ given in \eqref{eq.sx}\footnote{We abuse the
  notation referring $r(t) = a(1 - \ecc \cos u(t))$ and $r(u) = a (1 -
  \ecc \cos u)$ as equal (similarly $f(t)$ and $f(u)$).  Although
  formally we should label them differently, they are equivalent via the
  Kepler's equation.}.

Note that the introduction of $u$ as independent variable implies a
non-constant deformation in the angular component. More precisely, if
$y(t) = \frac{dx(t)}{dt}$, defining
\begin{equation}
 \label{eq.x2b}
 \beta (u) \coloneq x (u - \ecc \sin u)\ ,
\end{equation}
we obtain
\begin{equation*}
 \gamma (u) \coloneq  \frac{d\beta(u)}{d u} = \frac{r(u)}{a}y(u - \ecc
 \sin u)\ .
\end{equation*}

The ODE system \eqref{eq.spinbg} can now be integrated by any
classical numerical integrator (\cite{HairerNW1993}) without having to
solve the Kepler's equation \eqref{eq.tu} at each integration step. In
the following sections, we will use a Taylor's integrator
(\cite{JorbaZ2005}) that we briefly recall for self-consistency in
Appendix~\ref{sec.taylorint}. We have used Taylor's method because it
can produce solutions with high accuracy (say $10^{-30}$), since it
can easily increase the orders of the Taylor's expansions in order to
provide good enough trajectory values. These integrators are also used
to produce rigorous enclosures (\cite{BertzM1998}). Both features seem
to be important towards the goal of producing computer-assisted proofs
(progress towards this goal will be reported in \cite{CCGL20b}) or in
the study of phenomena at breakdown that require delicate calculations
not easy to make convincing (progress towards this goal will be
reported in \cite{CCGL20c}).

Even if not directly used in the present paper, we report in
Appendix~\ref{sec:variational} the variational equations associated to
\equ{eq.spinbg}. The variational equations with respect to coordinates
and parameters can be used for different purposes (see \cite{CCGL20b},
\cite{CCGL20c}), e.g. to compute the parameterization of invariant
structures, to get estimates based on the derivative of the flow, to
compute chaos indicators like the Fast Lyapunov Indicator
(\cite{froes}).

\section{The conformally symplectic spin-orbit map}
\label{sec:spin-map}

In this section, we introduce the notion of conformally symplectic
systems (Section~\ref{sec:CS}), we reduce the study of the spin-orbit
problem to a discrete map (Section~\ref{sec:map}) and we provide an
explicit expression of the conformally symplectic factor
(Section~\ref{sec:factor}).

\subsection{Conformally symplectic systems}\label{sec:CS}
Conformally symplectic systems are dissipative systems that enjoy the
remarkable property that they transform the symplectic form into a
multiple of itself.

The formal definition for $2n$-dimensional discrete and continuous
systems is the following.

\vskip.1in

Let $\mathcal{M} = U\times \mathbb{T}^n$ be the phase space with
$U\subseteq \mathbb{R}^n$ an open and simply connected domain with
smooth boundary. We endow the phase space $\mathcal{M}$ with the
standard scalar product and a symplectic form $\Omega$, represented by
a matrix $J$ at the point ${\underline z}$ acting on vectors
${\underline u},{\underline v}\in \mathbb{R}^{2n}$ as
$\Omega_{\underline z}({\underline u}, {\underline v}) = ({\underline
  u}, J({\underline z}){\underline v})$.

For the spin-orbit model, the symplectic form $\Omega$ is represented
by the constant matrix $J$ which takes the form
\begin{equation}
 \label{matrixJ}
 J =
 \begin{pmatrix}
  0 & 1 \\ -1 & 0
 \end{pmatrix}\ .
\end{equation}

\begin{defn}\label{def:conformallysymplectic}
A diffeomorphism $f$ on $\mathcal M$ is conformally symplectic if, and
only if, there exists a function $\lambda \colon
\mathcal{M}\to\mathbb{R}$ such that
\begin{equation}\label{CS}
f^* \Omega = \lambda\Omega\ ,
\end{equation}
where $f^*$ denotes the pull--back of $f$ (i.e., $f^*\Omega = \Omega
\circ f$) and $\lambda$ is called the conformal factor.
\end{defn}

In the following, we will consider a family $f _\mu \colon {\mathcal
  M}\rightarrow{\mathcal M}$ of mappings and we will call $\mu\in
{\mathbb R}$ the \emph{drift} parameter. Correspondingly, we will
replace \eqref{CS} by
\begin{equation*}
f_\mu^* \Omega = \lambda\Omega\ .
\end{equation*}

\vskip.1in

We notice that for $\lambda=1$, we recover the symplectic case.
Moreover, as remarked in \cite{CallejaCL11}, for $n=1$ any
diffeomorphism is conformally symplectic with a conformal factor that
might depend on the coordinates; in particular, when $\Omega$ is the
standard area, one has either $\lambda(x)=|\det(Df _\mu(x))|$ or
$\lambda(x)=-|\det(D f _\mu(x))|$.  When $n\ge 2$, it follows that
$\lambda$ is a constant (see, e.g., \cite{Banyaga02,CallejaCL11}).

\vskip.1in

Definition~\ref{def:conformallysymplectic} extends to continuous
systems by the use of the \emph{Lie derivative}.

\begin{defn}
We say that a vector field $X$ is a conformally symplectic flow if,
and only if, denoting by $L _X$ the Lie derivative, there exists a
function $\mu \colon \mathbb{R}^{2n} \rightarrow \mathbb{R}$ such that
\beq{flow} L_X\Omega = \mu \Omega\ .  \eeq
\end{defn}

\vskip.1in

Denoting by $\Phi_t$ the flow at time $t$ of the vector field $X$, we
observe that \eqref{flow} implies that
\[
(\Phi_t)^*\Omega=e^{\mu t}\Omega\ .
\]

\vskip.1in

In our applications we will consider a family of vector fields $X
_\sigma$ depending on a drift parameter $\sigma \in \R$ and we will
replace \equ{flow} by
\[
L_{X _\sigma}\Omega = \mu \Omega\ .
\]

In our applications we will also have to consider time-dependent
vector fields. A time dependent vector field $X(t)$ is conformally
symplectic when $L_{X(t) }\Omega = \mu(t) \Omega$.  It is not
difficult to show that $\Phi_a^b$, the diffeomorphism which takes
initial conditions at time $a$ to the position at time $b$ (hence
$\Phi_a^c = \Phi_b^c \circ \Phi_a^b$), satisfies
\begin{equation}\label{average}
(\Phi_a^b)^*\Omega  = e^{\int_a^b \mu(s) \, ds} \Omega \ .
\end{equation}
Note that \eqref{average} implies that if $X(t)$ is periodic of period
$T$, the conformal factor of $\Phi_a^{a+T}$ is independent of $a$ and
it is equal to the conformal factor of a vector field with a constant
$\bar{\mu} = \frac{1}{T} \int_a^{a+T} \mu(s) \, ds$.  This will be
useful in the justification of averaging (see
Appendix~\ref{sec:averaging}).

\vskip.1in

The spin-orbit models described by \eqref{eq.spinxy} and
\eqref{eq.avg-spinxy} are both conformally symplectic. Let us start to
show that the averaged system \eqref{eq.avg-spinxy} is conformally
symplectic. We write \eqref{eq.avg-spinxy} as the first order system
\begin{align*}
 \dot x&= y\\
 \dot y&= -\eps\Big(\frac{a}{r}\Big)^3\ \sin(2x-2f)-\mu (\dot x-\upsilon)\ ,
\end{align*}
  where $\mu = \dis\bar L(\ecc)$ and $ \upsilon = \frac{\bar
    N(\ecc)}{\bar L(\ecc)}$.  Hence, $\mu$ is the conformal factor
  and $ \upsilon $ is the drift parameter.  Denoting by $i_X$ the interior
  product and recalling that $\Omega=dy\wedge dx$, we have
\[
i_X\Omega=\dot y dx-\dot x dy=\Big[-\mu
  (y-\upsilon)-\varepsilon\Big(\frac{a}{r}\Big)^3\ \sin(2x-2f)\Big]\,
dx-y\, dy\ .
\]
Then, we have
\[
d(i_X\Omega)=-\mu\, dy\wedge dx=-\mu\Omega\ .
\]
Since
\[
L_X\Omega=i_X\, d\Omega+d(i_X\Omega)=d(i_X\Omega)\ ,
\]
we conclude that
\[
L_X\Omega=-\mu\Omega\ .
\]
A similar computation shows that the model described by
\eqref{eq.spinxy} is conformally symplectic.


\subsection{The spin-orbit map}\label{sec:map}
We introduce the Poincar\'e map associated to \equ{eq.spinxy} or
\equ{eq.avg-spinxy}, which allows one to reduce the continuous
spin-orbit problem to a discrete system. We denote by $G _\ecc$ the
flow at time $2\pi$ in the independent variable $u$ associated to the
equation of motion \equ{eq.spinbg} in the coordinates
$(\beta,\gamma)$:
\begin{equation}
 \label{eq.qe}
G _\ecc(\beta_0,\gamma_0)=
\begin{pmatrix}
 \beta(2\pi;\beta_0,\gamma_0,\ecc) \\
 \gamma(2\pi;\beta_0,\gamma_0,\ecc)
\end{pmatrix}
\end{equation}
with $\beta(2\pi;\beta_0,\gamma_0,\ecc) $ and $\gamma(2\pi;
\beta_0,\gamma_0,\ecc)$ denoting the solution at $u=2\pi$ with initial
conditions $(\beta _0,\gamma _0)$ at $u = 0$. If we set $G _\ecc =
(G^{(1)} _\ecc, G^{(2)} _\ecc)$, then the Poincar\'e map associated to
\equ{eq.spinbg} is
\begin{align*}
\bar\beta &= G^{(1)}_\ecc(\beta,\gamma) \\
\bar\gamma &= G^{(2)}_\ecc(\beta,\gamma)\ .
\end{align*}
The Poincar\'e map $P_\ecc$ at time $t = 2\pi$ associated to
\equ{eq.spinxy} is then given by the conjugacy
\begin{equation}
 \label{eq.pe}
 P _\ecc=\Psi _\ecc^{-1}\circ G _\ecc \circ \Psi _\ecc
\end{equation}
with the (time) change of coordinates from $(x,y)/(2\pi)$ to
$(\beta, \gamma)$ given by
\begin{equation}
 \label{eq.xy2bg}
 \Psi _\ecc = 2\pi\  \begin{pmatrix}
  1 & 0 \\ 0 & 1-\ecc
 \end{pmatrix}\ .
\end{equation}
Using the map $G _\ecc$ is very advantageous in numerical numerical
implementation because this allows us to avoid to dealing with the
Kepler's equation \eqref{eq.tu}.  On the other hand, the map $P _\ecc$
is appropriate for physical interpretations, and the close and explicit
relation among them \eqref{eq.pe} allows to choose the most
advantageous one for the task at hand.

\subsection{The conformally symplectic factor}
\label{sec:factor}
Our next task is to find an explicit form for the conformally
symplectic factor $\lambda$ of the Poincar\'e map associated to
\equ{eq.spinxy} or equivalently \equ{eq.spinbg}.

\vskip.1in

By the Jacobi-Liouville Theorem, the determinant of the differential
of the Poincar\'e map is obtained by integrating the trace of the
Jacobian matrix associated to \eqref{eq.spinxy}. Since the determinant
and the trace of a matrix are invariant under a change of basis, their
values are the same if we compute the system \eqref{eq.spinbg} via the
Jacobian with elements $a_{ij}$ given by
\[
 \begin{aligned}
  a _{11}&= 0, & a _{21}&= -2\eps \frac{a}{r(u)}c(\beta;u,\ecc) \ , \\
  a _{12}&= 1, & a _{22}&= \ecc \frac{a}{r(u)} \sin u - \dis \biggl(\frac{a}{r(u)}\biggr)^5\ .
 \end{aligned}
\]
This leads to the following expression for the conformal factor:
\begin{equation}
\label{conformalfactor}
 \lambda = \exp \left( \int _0^{2\pi} \left(\ecc \frac{a}{r(u)}\sin u - \dis
    \left(\frac{a}{r(u)}\right)^5 \right) \, du \right)\ .
\end{equation}

It is remarkable that the integral in \eqref{conformalfactor} can be
computed analytically in terms of $\dis$ and $\ecc$. To this end, we
need the following preliminary result.

\begin{lem}
\label{lem.integral}
 If $\ecc\in [0, 1)$ and $r = a (1 - \ecc \cos u)$, then
 \[
  \int _0^{2\pi} \left(\frac{a}{r(u)}\right)^5 \, du = \pi \frac{3\ecc^4
    +24 \ecc^2 +8}{4 (1 - \ecc^2)^{9/2}}\ .
 \]
\end{lem}
\begin{proof}
 Since $\ecc \in [0,1)$, then $(\frac{a}{r})^5$ has no poles in the
   unit circle. By the change of variables $z = \exp(\I u)$ and a
   straightforward application of the Residue Theorem, we obtain:
 \[
  \int _0^{2\pi} \biggl(\frac{a}{r(u)} \biggr)^5 \, du = 2 \pi \sum
  _{\substack{ w \text{ sing.} \\ \lvert w \rvert < 1}}
  \Res\biggl(\frac{2^5z^5}{(2z-\ecc z^2-\ecc)^5}, w\biggr)\ ,
 \]
 where $\Res$ denotes the residue of a holomorphic function. To
 compute the residue, we need to evaluate the poles $\alpha _\pm$
 which are given by $\alpha _\pm = \ecc^{-1}(1 \pm \sqrt{1 -
   \ecc^2})$. Since $|\alpha _- | < 1$, then by an explicit
 computation, we get:
\begin{align*}
 \int _0 ^{2\pi} \biggl(\frac{a}{r(u)} \biggr)^5\, du &=
 2 \pi \frac{1}{4!} \lim_{z \to \alpha_-} \frac{d^4}{dz^4}
(\frac{-2^5z^4}{\ecc^5(z-\alpha_+)^5}) \\
 &= 2\pi \lim_{z \to \alpha_-}\frac{32 \left(\alpha_+^4+16 \alpha_+^3 z+36
\alpha_+^2 z^2+16 \alpha_+ z^3+z^4\right)}{\ecc^5 (\alpha_+-z)^9}
\\
 &= \pi \frac{3 \ecc^4+24 \ecc^2+8}{4 \left(1-\ecc^2\right)^{9/2}}. \qedhere
\end{align*}
\end{proof}

The above result leads to the following Corollary, which gives an
explicit form of the conformal factor of the spin-orbit model
described by \eqref{eq.spinxy}.

\begin{cor}
 \label{cor.simplfactor}
 The conformally symplectic factor of the $2\pi$-time map of the
 spin-orbit problem with tidal torque given by the system
 \eqref{eq.spinxy} has the following expression:
 \[
   \lambda = \exp \biggr(-\dis \pi \frac{3 \ecc^4+24 \ecc^2+8}{4
     \left(1-\ecc^2\right)^{9/2}}\biggl)\ .
 \]
\end{cor}

\begin{rems}
 \begin{enumerate}
\renewcommand*{\theenumi}{\emph{\roman{enumi}}}
\renewcommand*{\labelenumi}{(\theenumi)}
  \item Note that the result in Corollary~\ref{cor.simplfactor} is
    also valid under the change of time given in \eqref{eq.tu}. This
    means that the $2\pi$-time map associated to the system
    \eqref{eq.spinbg} has the same symplectic factor, since the
    determinant is invariant under a change of basis.
  \item The conformally symplectic factor $\lambda$ can be either
    contractive, expansive or neutral. The value $\lambda$ has a clear
    dynamical interpretation: at each $2\pi$-interval of time, the
    values move  $\lambda$ far away from unity. We remark that in
    this work we are interested to the contractive case, namely $\dis
    > 0$.
 \end{enumerate}
\end{rems}



\section{The computation of invariant attractors}
\label{sec:attractors}

We consider the model described by equation \eqref{eq.spinxy} and we
provide an algorithm (see Section~\ref{sec:algorithm}) for the
construction of KAM invariant attractors. The algorithm relies on the
fact that, starting from an initial approximate solution, one can
construct a better approximate solution. A possible choice for the
initial approximate solution is presented in
Section~\ref{sec:approx}. A validation of the goodness of the solution
is considered in Section~\ref{sec.accuracy-tests}, which provides some
accuracy tests.  The construction of the invariant attractors through
the implementation of the algorithm described in
Section~\ref{sec:algorithm} requires some technical procedures,
precisely multiple precision arithmetic (see
Appendix~\ref{sec:multiple}) and parallel computing (see
Section~\ref{sec:parallel}).  Examples of the application of
Algorithm~\ref{alg.newton} are given in
Section~\ref{sec:implementation}.

\subsection{An algorithm for constructing invariant attractors}
\label{sec:algorithm}

Invariant attractors for the map $P_\ecc$ defined in \eqref{eq.pe}
associated to the dissipative spin-orbit equation \equ{eq.spinxy}
can be obtained by implementing an efficient algorithm based on
Newton's method; this algorithm is also used to give a rigorous
proof of invariant tori through KAM theorem, see \cite{CCGL20b},
as well as to give accurate bounds on the breakdown threshold, see
\cite{CCGL20c}.

To introduce the algorithm, we need to fix the frequency $\omg$, that
we are going to choose sufficiently irrational, see
Definition~\ref{def:diophantine}, and we need to introduce invariant
KAM attractors, see Definition~\ref{def:inv}.

\begin{defn} \label{def:diophantine}
The number $\omega\in\R$ is said Diophantine of class $\tau$ and
constant $\nu$ for $\tau\geq 1$, $\nu>0$, and briefly denoted as
$\omega \in \mathcal{D}(\nu, \tau)$, if the following inequality
holds:
\beq{DC} |\omega\ k-q|^{-1}\leq \nu^{-1}|k|^\tau \eeq
for $q\in\Z$, $k\in\Z\backslash\{0\}$.
\end{defn}

We remark that the sets of Diophantine numbers as in
Definition~\ref{def:diophantine} is such that their union over $\nu>0$
has full Lebesgue measure in $\R$.\\

Next, we define as follows an invariant attractor with frequency
$\omg$ satisfying \equ{DC}.

\begin{defn} \label{def:inv}
Let $P _\ecc \colon \M\rightarrow\M$ be a family of conformally
symplectic maps defined on a symplectic manifold $\M \subset \R \times
\T $ and depending on the drift parameter $\ecc$.  A KAM attractor
with frequency $\omega$ is an invariant torus described by an
embedding $K _p \colon \T\rightarrow\M$ and a drift parameter $e_p$,
satisfying the following invariant equation for $\theta\in\T$:
\beq{invariance} P_{\ecc_p}\circ K_p(\theta) = K_p(\theta+\omega)\ .
\eeq
\end{defn}

We will often write \equ{invariance} in the form
\[
P_{\ecc _p}\circ K_p = K_p\circ T_\omg\ ,
\]
where $T_\omg$ denotes the shift function by $\omg$, i.e., $T
_\omg(\theta) = \theta + \omg$.

The solutions of the invariance equation \eqref{invariance} is unique
up to a shift.  For all real $\alpha$, if $\hat K _p(\theta) \coloneq
K _p(\theta + \alpha)$, then $(\hat K _p, \ecc _p)$ is also a solution
of \eqref{invariance}.  Note that all these solutions parameterize the
same geometric object. In \cite{CallejaCL11} there is a simple
argument showing that this is the only source of lack of local
uniqueness.

\begin{rem}
  It is easy to see -- even in the integrable case -- that to obtain
  an attractor with a specific frequency we need to adjust the drift
  parameter.

  Hence, repeating the calculation several times, we obtain the drift
  as a function of the frequency.  One can invert this --
  one-dimensional -- function and obtain the frequency as a function
  of the drift, so that the two are mathematically equivalent, see
  Figure~\ref{fig:rotations}.

  For astronomers, the frequency is directly observed and it is
  natural to think of using the measurements of the frequency to
  obtain values of the drift.

  Theoretical physicists may prefer that the values of the drift are
  known and that one predicts the frequency.

  Both points of view are mathematically equivalent modulo inverting a
  1-D function. In astronomy, since there is little a-priori
  information on the values of the elastic properties of the satellites, it
  seems more natural to study the drift as function of the
  frequency. On other physical applications, where the values of the
  model are known from the start, the other point of view may be
  preferable.

  One small technical problem (that can be solved) is that the
  function is not defined for all values of the frequency. The theory
  only establishes for a set of large measure. Repeatedly, not all the
  values of the drift parameter lead to a system that has a rotational
  attractor.
\end{rem}

\medskip

The starting point of the iterative process in the spin-orbit problem
is then $(K,\ecc)$, an approximate solution of the invariance equation
\equ{invariance},
\begin{equation} \label{approx}
P_\ecc \circ K(\theta) -
K(\theta+\omg)=E(\theta)\ ,
\end{equation}
where the \emph{``error''} $E$ is thought of as small (making precise
the notion of small will require the introduction of norms).

Algorithm~\ref{alg.newton} below takes the pair $(K , \ecc)$ and
produces another approximate solution $(\tilde K,\tilde \ecc)$, which
satisfies \equ{invariance} up to an error whose norm is quadratically
smaller with respect to $E$ (again, making all this precise requires
introducing norms).

We note that all the steps are rather explicit operations taking
derivatives, shifting and performing alebraic operations. The most
delicate steps are \ref{alg.newton-cohom1}, \ref{alg.newton-cohom2},
which involve solving cohomology equations and step
\ref{alg.newton-ls} which involves solving a $2 \times 2$ linear
equation.  The assumption of invertibility of this explicit $2\times
2$ matrix is a non-degeneracy assumption that takes the place of the
classical twist condition.

The Algorithm~\ref{alg.newton} is based on that described in
\cite{CallejaCL11} and adapted for the spin-orbit problem. Even if,
for the sake of simplicity, we only present the algebraic operations
in the recipe, in \cite{CallejaCL11} there are geometric
interpretations that motivate the steps.

\begin{alg}[Newton's method for finding a torus in the spin-orbit problem]
 \label{alg.newton}
 \
 \begin{enumerate}
\setlength{\itemsep}{.8em}
\renewcommand*{\theenumi}{\emph{\arabic{enumi}}}
\renewcommand*{\labelenumi}{\theenumi.}
  \item [$\star$] \texttt{Inputs:} $J$ as in \eqref{matrixJ}, $\omg$ a
    fixed frequency, an initial embedding $K\colon \mathbb{T}
    \rightarrow \mathbb{T} \times \mathbb{R}$, access to the
    $2\pi$-time flow map $G _\ecc$ of \eqref{eq.spinbg} for fixed
    values $\eps$ and $\dis$, change of coordinates depending on
    $\ecc$, $\Psi _\ecc \equiv 2\pi \left(
  \begin{smallmatrix}
   1 & 0 \\ 0 & 1-\ecc
  \end{smallmatrix}
\right)$, and conformally symplectic map $P _\ecc \equiv \Psi^{-1}
_\ecc \circ G _\ecc \circ \Psi _\ecc $.
  \item [$\star$] \texttt{Output:} New $K$ and $\ecc$ satisfying the
    invariance equation \eqref{invariance} up to a given tolerance.
  \item [$\star$] \texttt{Notation:} If $A$ is defined in
    $\mathbb{T}$, $\overline A \coloneq \int _{\mathbb{T}} A$ and $A^0
    \coloneq A - \overline A$.
  \item \label{alg.newton-error} $E \gets P _ \ecc \circ K - K \circ T
    _\omg$, \newline $E _1 \gets E _1 - \mathtt{round}(E _1)$.
  \item $\alpha \gets D K$.
  \item \label{alg.newton-N} $N \gets (\alpha ^t \alpha) ^{-1}$.
  \item \label{alg.newton-M} $M \gets
  \begin{bmatrix}
   \alpha & J ^{-1} \alpha N
  \end{bmatrix}$.
  \item $\widetilde E \gets (M ^{-1}\circ T _\omg) E$.
  \item $\lambda$ from the Corollary~\ref{cor.simplfactor}.
  \item \label{alg.newton-PSA} $P \gets \alpha N$, \\ $S \gets (P
    \circ T _\omg)^t D P _ \ecc \circ K J ^{-1} P$, \\ $\widetilde A
    \gets M ^{-1} \circ T _\omg D _\ecc P _\ecc \circ K$.
  \item \label{alg.newton-cohom1} $(B _a)^0$ solving $\lambda (B _a)
    ^0 - (B _a)^0 \circ T _ \omg = - (\widetilde E _2)^0$, \\ $(B
    _b)^0$ solving $\lambda (B _b) ^0 - (B _b)^0 \circ T _ \omg = -
    (\widetilde A _2)^0$.
  \item \label{alg.newton-ls} Find $\overline W _2$, $\sigma$ solving
    the linear system
  \begin{equation*}
   \begin{pmatrix}
    \overline {S} & \overline{S(B _b)^0} +
\overline{\widetilde A _1} \\
    \lambda - 1 &  \overline{\widetilde A _2}
   \end{pmatrix}
   \begin{pmatrix}
     \overline W _2 \\ \sigma
   \end{pmatrix} =
   \begin{pmatrix}
     - \overline{\widetilde E _1} - \overline{S(B _a)^0} \\
     - \overline{\widetilde E _2}
   \end{pmatrix}\ .
  \end{equation*}
  \item $(W _2)^0 \gets (B _a)^0 + \sigma (B _b)^0$.
  \item $W _2 \gets (W _2)^0 + \overline{W}_2$.
  \item \label{alg.newton-cohom2} $(W _1)^0$ solving $(W _1) ^0 - (W
    _1)^0 \circ T _ \omg = - (S W _2)^0 - (\widetilde E _1)^0 -
    (\widetilde A _1)^0 \sigma$.
  \item \label{alg.newton-correction} $K \gets K + M W$, \\
  $\ecc \gets \ecc + \sigma$.
 \end{enumerate}
\end{alg}
Algorithm~\ref{alg.newton} needs some practical remarks:
\begin{itemize}
 \item Because of the periodicity condition $K (\theta + 1) = K
   (\theta ) + \left(
 \begin{smallmatrix}
  1 \\ 0
 \end{smallmatrix} \right)$, one can always define the periodic map
 $\widetilde K(\theta) \coloneq K(\theta) - \left(
 \begin{smallmatrix}
  \theta \\ 0
 \end{smallmatrix} \right)$ which, generically, admits a Fourier series.
 Then, one obtains that
 \[
  K\circ T _\omg = \widetilde K \circ T _\omg + \left(
  \begin{smallmatrix}
   \omg \\ 0
  \end{smallmatrix} \right).
 \]
 \item The function $E$ in step \ref{alg.newton-error} must perform
   the subtraction in the first component in $\mathbb{T}$. For a
   numerical implementation that can be fulfilled by the assignment $E
   _ 1 \gets E _1 - \round(E _1)$, where $\round$ returns the nearest
   integer value of its argument. Such a function is commonly provided
   in almost all programming languages.
 \item The matrix $M$ in step \ref{alg.newton-M} is unimodular, which
   allows to get an easy inverse matrix expression.
 \item The quantities $D P _\ecc$ and $D _\ecc P _\ecc$ in step
   \ref{alg.newton-PSA} are needed to compute the directional
   variational flow of $\Phi$, which can be done automatically using
   the explanations in Section~\ref{sec.vars}.
 \item The stopping criterion is either that $\|E\|$ or $\max\{ \| MW
   \|, | \sigma |\}$ is smaller than a prefixed tolerance.
\end{itemize}
A common numerical representation for periodic mappings is a Fourier
series in the inputs $K_1$ and $K _2$, which gives us a representation
of $K\equiv (K _1, K _2)$.  In such a representation, we can use the
Fourier transform, or its numerical version via Fast Fourier Transform
(FFT) algorithm.

Therefore any periodic mapping, say $f$, admits two representations,
namely in points (or table of values) and in Fourier coefficients. The
first one is just the values $(\widecheck f _k) _{k = 0}^{n-1}$ of the
mapping in an equispaced mesh in $[0,1)$ of size $n$. The second one
  is obtained by the Inverse of the Fast Fourier Transform (IFFT),
  denoted by $(\widehat f _k) _{k = 0}^{n-1}$.  Notice that, because
  the function is assumed to be real-valued, the two representations
  can have the same size, i.e. $n$ real values.

Depending on the step in Algorithm~\ref{alg.newton}, it may be better
to use one representation or the other.  For instance, $P _\ecc \circ
K$ in \ref{alg.newton-error}, and $D P _\ecc \circ K$, $D _\ecc P
_\ecc \circ K$ in \ref{alg.newton-PSA} are better when $K$ is in a
table of values, although an ODE version of \eqref{eq.spinxy} in terms
of Fourier coefficients can be considered.

On the other hand, the composition with $ T _\omg$ and the solution of
the cohomological equations in \ref{alg.newton-cohom1} and
\ref{alg.newton-cohom2} are easier if $K$ is in Fourier
series. Indeed, these equations can be solved in Fourier coefficients,
using the following result whose proof is straightforward.

\begin{lem}
\label{lem.cohoms}
 Let $\eta(\theta) = \sum _k \eta _k \exp(2\pi \I k \cdot \theta)$ and
 let $\omg $ be irrational. Then:
 \begin{enumerate}
  \renewcommand{\theenumi}{\roman{enumi}}
  \renewcommand{\labelenumi}{(\theenumi)}
  \item If $\eta _0 = 0$, then $\phi \circ T _\omg - \phi = \eta$
    has solution $\phi(\theta) = \sum _k \phi _k \exp(2\pi \I k \cdot
    \theta)$ with
   \[
    \phi _k =
    \begin{cases}
     \frac{\eta _k}{\exp(2\pi \I k \cdot \omg) - 1} & \text{if } k \ne 0, \\
     0 & \text{otherwise.}
    \end{cases}
   \]
  \item If $\lambda$ is not a root of unit, then $\phi \circ T _\omg
    - \lambda \phi = \eta$ has solution with coefficients
   \[
    \phi _k = \frac{\eta _k}{\exp(2\pi \I k \cdot \omg) - \lambda} \ .
   \]
 \end{enumerate}
\end{lem}

\subsection{Initial approximation of the invariant curve}
\label{sec:approx}


Repeated application of the Newton's method from
Algorithm~\ref{alg.newton} produces a very accurate solution provided
that one can get a good enough initial approximation.  In this section
we address the problem of producing such an initial
approximation. Although the methods are rather general, we are going
to give the results for the cases of study in this paper. We select
the following two frequencies with good Diophantine conditions
belonging to the class $\mathcal{D}( \frac{2}{3-\sqrt{5}}, 1)$, see
Definition~\ref{def:diophantine}:
\begin{align}
 \label{eq.omggold}
  \omg _1 &= \gamma _g^+
  \intertext{and}
 \label{eq.omg}
  \omg _2 &= 1 + \frac{1}{2 + \gamma _g^-} \ ,
\end{align}
where we define $\gamma _g^\pm = \frac{\sqrt{5}\pm 1}{2}$.

One method to provide an initial approximation is, of course, to
continue from an integrable case that can be solved
explicitly. Another method is to do an easy calculation of an
approximate torus. Since the system is dissipative, the torus, if it
exists, will be an attractor.

\subsubsection{A continuation method from the integrable case}
\label{sec:integrable}
Let us start to analyze the averaged problem for which we look for the
drift starting from the integrable case.  Fix the frequency $\omg$ and
set $\eps=0$ in \equ{eq.avg-spinxy}; then, the drift $\ecc$ can be
chosen so that
\begin{equation*}
\frac{\bar N(\ecc)}{\bar L(\ecc)} = \omg\ .
\end{equation*}
The above equation provides the eccentricity as a function of the
frequency.  In fact, as noticed in \cite{ARMA}, for $\dis \neq 0$ the
solution of \equ{eq.avg-spinxy} can be written as
\[
x(t) = x(0) + \frac{\bar N(\ecc)}{\bar L(\ecc)} t + \frac{1-\exp({-\dis
    t})}{\dis}\ \biggl(\dot x(0) - \frac{\bar N(\ecc)}{\bar L(\ecc)}\biggr)\ ,
\]
which shows that $\dot x=\frac{\bar N(\ecc)}{\bar L(\ecc)}$ is a
global attractor for the unperturbed, purely dissipative case
$\eps=0$. Setting $y=\dot x$, we can select as initial starting
condition
\begin{equation}
 \label{eq.starting-iter-point}
 x(0) = 0 \qquad \text{and} \qquad y(0)={{\bar N(\ecc)}\over {\bar L(\ecc)}}\ .
\end{equation}
We remark that, despite the use of the averaged version of the
spin-orbit problem given in \eqref{eq.avg-spinxy}, to approximate an
initial guess of the eccentricity in terms of the frequency, we can
also use the non-averaged spin-orbit problem \eqref{eq.spinxy} to
provide an approximated eccentricity, see Section~\ref{sec:direct}.

\subsubsection{Direct iteration} \label{sec:direct}
We propose a method different from Section~\ref{sec:integrable}, which
takes advantage of the fact that the torus, if it exists, is an
attractor.  We do not need to consider the average equation and we can
start from the model \equ{eq.spinxy}.  Hence, we start by selecting a
set of points at random; after a transient number of iterations (e.g.,
choose the transient as the inverse of the dissipation multiplied by a
convenient safety factor), we expect that the orbit is close to the
attractor. Then, we assess whether indeed this orbit has a rotation
number.

Since not all the attractors of dissipative maps are rotational
orbits, not all of them should have a rotation number.  Of course,
there may be situations where the orbit is a chaotic attractor that
happens to have a rotation number.

Given the importance of the rotation number, there are quite a number
algorithms to compute it, e.g.  \cite{LaskarFC92, Athanassopoulos1998,
  Laskar99, AlsedaLM2000, Laskar99, SearaV06, GomezMS10a, GomezMS10b,
  Simo2010}.  We have used the method in \cite{Sander2017} which
speeds the convergence to the rotation number. We will present more
details of the calculation in \cite{CCGL20b}. We note that the method
in \cite{Sander2017} gives a very good indication of the existence of
an invariant circle. In \cite{Sander2017} it is shown that if there is
a smooth invariant circle, the convergence of the method to a rotation
number is very fast. Hence the fast convergence of the algorithm is a
reasonably good evidence of the existence of a rotational torus.  Of
course, the convergence of the Newton's method started in this guess
is a much stronger validation of the correctness of the guess.

Figure~\ref{fig:rotations} compares the two approaches explained here
and in Section~\ref{sec:integrable}. The first one is straightforward,
since it consists in plotting the function $\bar N(\ecc)/\bar L(\ecc)$
in terms of the eccentricity $\ecc$. The second one requires a little
bit more effort, since it needs to numerically integrate
\eqref{eq.spinxy} to compute the rotation number.

\begin{figure}[ht]
  \label{fig:rotations}
 \input{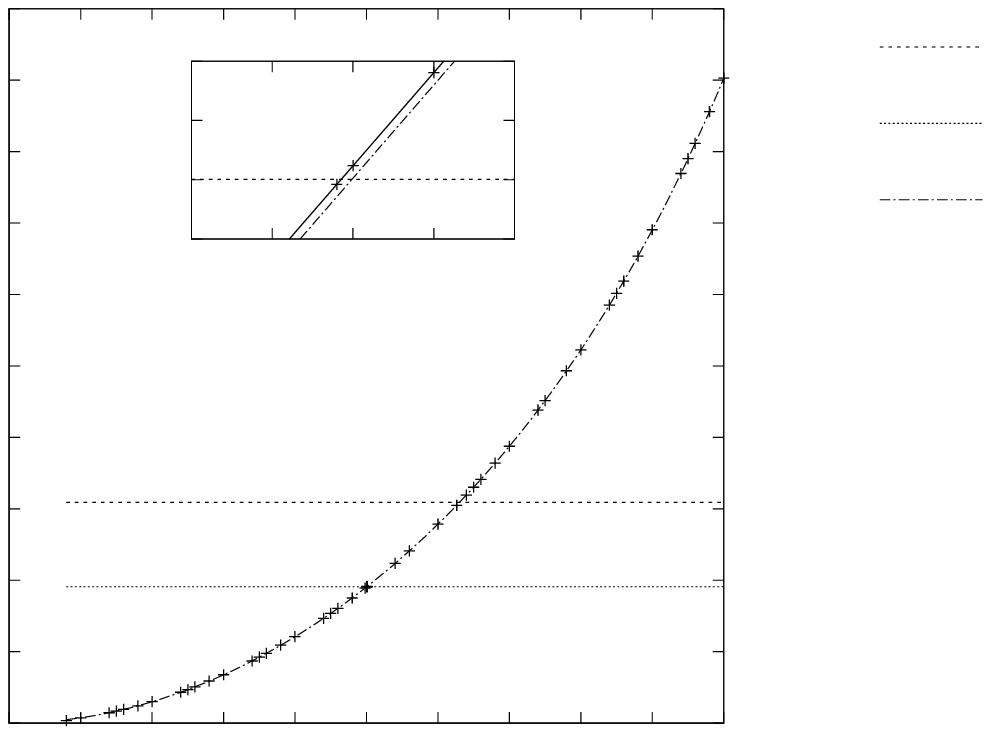}
\hglue-.4cm
\vglue-.4cm
 \caption{Eccentricity versus the rotation number denoted by
   \texttt{+} of the system \eqref{eq.spinxy} with $\eps = 10^{-4}$
   and $\dis=10^{-5}$. Frequencies as in \eqref{eq.omggold} and
   \eqref{eq.omg}. The small window is just a zoom-in near to $\omg
   _2$ which shows that the averaged quantity $\bar N(\ecc) /\bar
   L(\ecc)$ approaches the non-averaged rotation number.}
 \label{fig:rotations}
\end{figure}

\subsubsection{Initial approximation for the embedding}

Once we have fixed an initial guess for the eccentricity and an
initial starting point, we can proceed to get an initial guess of the
embedding $K$ which is needed as input in the
Algorithm~\ref{alg.newton}. To this end, we first perform a
preliminary transient of iterations of $G _\ecc$, defined in
\eqref{eq.qe}, at the initial point $\Psi _\ecc (x(0), y(0)/(2\pi))$,
with $(x(0),y(0))$ given in \eqref{eq.starting-iter-point}. After
that, we can follow the steps in Algorithm~\ref{alg.inv-curve} to get
the initial guess for $K$ for the Algorithm~\ref{alg.newton}.

\begin{alg}[Invariant curve approximation]
\label{alg.inv-curve}
 \
 \begin{enumerate}
\setlength{\itemsep}{.8em}
\renewcommand*{\theenumi}{\emph{\arabic{enumi}}}
\renewcommand*{\labelenumi}{\theenumi.}
  \item [$\star$] \texttt{Inputs:} Points $(\beta _k, \gamma
    _k)_{k=0}^{n-1}\subset [0,2\pi) \times \mathbb{R}$ from iteration
    of the $2\pi$-time Poincar\'e map of \eqref{eq.spinbg} with
    parameter values $\eps, \dis$, and $\ecc$ and number of Lagrange
    interpolation points $2j\in \mathbb{N}$.
  \item [$\star$] \texttt{Output:} Initial embedding $K$ for
    Algorithm~\ref{alg.newton} with a given mesh size $n _\theta$.
  \item Sort $(\beta _{i _k}, \gamma _{i _k}) _{k =0}^{n-1}$ such that
    $\beta _{i _ 1} \leq \dotsb \leq \beta _{i _n}$.
  \item Mesh $\overline \beta _k = 2\pi k / n _\theta$ with $k = 0,
    \dotsc, n _\theta-1$.
  \item For each $k = 0, \dotsc, n _\theta - 1$, let $\overline \gamma
    _k$ be the Lagrange interpolation centered at $(\beta _{i
    _k},\gamma _{i _k})$ with $2j$ points $\gamma _{i _{k}-j \pmod
    n},\dotsc, \gamma _{i _ {k}+j \pmod n}$ and their respective
    abscissae.
  \item Return the table of values $K \equiv (\Psi _\ecc^{-1}
    (\overline \beta _k, \overline \gamma _k)) _{k=0}^{n_\theta -1}$
    with $\Psi _\ecc$ from \eqref{eq.xy2bg}.
 \end{enumerate}
\end{alg}

\subsection{Accuracy tests}\label{sec.accuracy-tests}
We are now looking for $K \colon \mathbb{T} \rightarrow \mathbb{T}
\times \mathbb{R}$ and the parameter $\ecc$ so that the invariance
equation \eqref{invariance} is verified, for $P _\ecc$ given in
\eqref{eq.pe} and a fixed frequency $\omg$ like in \eqref{eq.omggold}
or \eqref{eq.omg}.

In this process, we have three main sources of error that affect the
result.
\begin{enumerate}
\renewcommand*{\theenumi}{\bf E\arabic{enumi}}
\renewcommand*{\labelenumi}{\theenumi.}
\setlength{\itemsep}{.5em}
 \item \label{err1} The error of the invariance condition on the table
   of values.  This error is controlled by the Newton's procedure.
 \item \label{err2} The error on the integration. This error is
   controlled by the numerical integrator when we request absolute and
   relative tolerances.
 \item \label{err3} The error in the discretization.  To control this
   source of errors we need first to estimate it, and then to be able
   to change the mesh when the error is too large.
\end{enumerate}
Let us address the error coming from \eqref{err3}. Let $\mathcal{A}
\subset \mathbb{T}$ be the set of points corresponding to the table of
values used in the algorithm. In the case of a Fourier representation
of size $n _\theta$, then $\mathcal{A} = \{k / n_\theta \}_{k =
  0}^{n_\theta -1}$ is an equispaced mesh of $[0,1)$. After some
  iterations of the Newton's Algorithm~\ref{alg.newton}, we obtain a
  set of values $\{K (\theta _i)\} _{\theta _i \in \mathcal{A}}$ and
  an eccentricity $\ecc$ satisfying the invariance equation in a mesh
\begin{equation}
\label{eq.errinv}
 \max _{\theta _i \in \mathcal{A}} \| P _{\ecc} \circ K (\theta _i) -
 K (\theta _i + \omg) \| < \delta\ ,
\end{equation}
where $\delta$ is a fixed tolerance, e.g. $\approx 10^{-11}$ in double
precision. Note that we are not fixing the norm in \eqref{eq.errinv}
which typically can be the sup-norm, the analytic norm, etc.

Let us now define $\delta ^\ast$ as
\begin{equation*}
 \delta^\ast = \max _{\theta \in \mathbb{T}} \| P _{\ecc} \circ K
 (\theta ) - K (\theta + \omg) \|\ .
\end{equation*}
The computation of $\delta ^\ast$ is in general difficult and we
suggest two standard heuristic alternatives.

The first option is very fast: it consists in looking at the norm of
some of the ``last'' Fourier coefficients and using it as an estimate
for the truncation error of the series. Once the Newton's iteration
has converged on a given mesh, we check the size of these
coefficients. If one of them is larger than a prescribed threshold, we
assume that the interpolation error is too big, and we increase the
number of Fourier modes in the direction of these large coefficients.

The second option is to evaluate the error in \eqref{eq.errinv} on a
set of values $\widetilde {\mathcal{A}} \subset \mathbb{T}$ different
from $\mathcal{A}$.  One can use a thinner set $\widetilde
{\mathcal{A}} $ to produce a better estimate of the invariance $\delta
^\ast$. This procedure can be computationally expensive. An easier
alternative is to consider $\widetilde {\mathcal{A}} $ with the same
number of points as $\mathcal{A}$. For instance,
\begin{equation*}
 \widetilde {\mathcal{A}}  = \mathcal{A} + \upsilon
\end{equation*}
with $\upsilon$ equal to one half of the distance between points of
$\mathcal{A}$ in the direction $\theta _i \in \mathcal{A}$. In the
Fourier case, $\widetilde {\mathcal{A}} = \{(k + 0.5) / n _\theta \}
_{k = 0}^{n _\theta - 1}$ should be enough.

Hence, we have a new mesh $\widetilde {\mathcal{A}} $ which is
interlaced with the initial mesh $\mathcal{A}$. Then, we check that
\begin{equation}
\label{eq.errinv2}
 \max _{\theta _i \in \widetilde {\mathcal{A}} } \| P _{\ecc} \circ K
 (\theta _i) - K (\theta _i + \omg) \| < \delta\ .
\end{equation}
If this test is not satisfied, we add more Fourier coefficients and we
go back to the Newton's iteration given by the
Algorithm~\ref{alg.newton}.  If the test is satisfied, we can either
stop and accept the solution or check it again with a thinner mesh.

The key is then to avoid checking with thinner meshes during the
computation as much as possible, because it is too costly, and to do
just a single check at the end to ensure the accuracy.

\subsection{Implementation of the algorithm}

Once Algorithm~\ref{alg.newton} is coded so that it becomes a sequence
of arithmetic operations (and transcendental functions), it is almost
easy to make it run in extended precision arithmetic, see
Appendix~\ref{sec:multiple}.

There are, however some caveats:
\begin{itemize}
\item One loses the hardware support.
\item The hardware optimized libraries have to be replaced by hand
  coded libraries.  Notably, one cannot use BLAS, LAPACK, or FFTW and
  they have to be substituted by explicit algorithms. We have used our
  own implementations.
\item In iterative processes, one has to choose the stopping criteria
  appropriately. As standard, one writes the stopping criteria as a
  power of \emph{Machine epsilon}.
\item For us, the most important point is that, in order to achieve
  high accuracy of the ODE integration with a reasonable step in a
  reasonable amount of time, one needs a high order method.

We have used the Taylor's method, which is based on computing the
Taylor's expansion of the solution of the equation to a very high
order, see Appendix~\ref{sec.taylorint}.

The paper \cite{JorbaZ2005} presents a very general purpose generator
of Taylor's integrators based on Automatic Differentiation. If one
specifies (in a very simple format) a differential equation, the
program \emph{taylor} (supplied and documented in \cite{JorbaZ2005})
generates automatically an efficient Taylor solver written in C.  The
user can select whether this Taylor solver uses standard arithmetic or
extended precision arithmetic (either {\tt GMP} or {\tt MPFR}). It is
important that Taylor's methods can work well with different versions
of the arithmetic.

One important product of the Taylor's integrator \cite{Gimeno2021} is
that we obtain very efficient solvers of the variational equations.
We will report on them in Appendix~\ref{sec:variational}, since they
are a natural extension of the Taylor's integrator.  We note that we
will not use them in the numerical experiments of this paper, but we
will use them in \cite{CCGL20b}. We remark that in mission design,
they appear in the method of \emph{differential corrections}.
  \end{itemize}

\subsubsection{Using profilers to detect bottlenecks}
\label{sec:profiler}
The computation of invariant tori with the Newton method that we
propose is remarkably fast when the initial guess is good
enough. However, the continuation to the breakdown requires to
increase the Fourier modes and then the computational cost will
increase in proportion.

The key step in a continuation process is the correction of the
solution for the new parameter values are being continued. In our
case, it is the Newton step. We have used a C profiler in a single
continuation step with multiprecision of 55 digits to realize which
are the most CPU-time consuming parts. We did it for different values
of $\eps \in \{10^{-4}, 2\cdot 10^{-4}\}$, $\dis \in \{10^{-3},
10^{-6}\}$, $\omg\in \{\omg _1, \omg _2\}$, and $N \in \{128,256\}$
number of Fourier modes getting, in all of them, similar results.

In average around 98.1\% was dedicated to the Newton step correction,
inside this step around 97.6\% was for the evaluation of the ODE (and
its variationals) in the ODE integrator and the correction of the
integration stepsize. Inside of it, around 31\% was for the addition
of jet transport elements, 25\% for multiplication, 18\% for
assignments, and 6.5\% for scalar multiplications.

As consequence, we conclude that the FFT, the solvers of cohomological
equations and the shiftings in Algorithm~\ref{alg.newton} are
irrelevant in terms of CPU-time as well as the memory allocation. The
second conclusion is that the ODE integration is the crucial
part. This fact will be exploited and detailed in
Section~\ref{sec:parallel}.

\subsection{Parallelization}
\label{sec:parallel}

There are several operations in the Algorithm~\ref{alg.newton} that
are fully independent to each other, such as those steps which are
done in a table of values of $\theta$ and the solution of the
cohomological equations using the Lemma~\ref{lem.cohoms}.

The use of a profiler shows that the main bottleneck, in terms of
CPU-time usage, is the ODE integration involved in $P _\ecc$, given in
\eqref{eq.pe}, and its first order directional derivatives.

A simple concurrent parallelization for each of the different
numerical integrations (previously ensuring that there is non-shared
memory between the threads) shows an evident speed-up with respect to
non-concurrent versions. In our case, we run the code with
multiprecision arithmetic, in particular with \texttt{MPFR}, and we
must be sure that each of the parallelized parts work correctly with
the multiprecision. In the case ot \texttt{MPFR} we must initialize
the precision and the rounding mode for each of the different CPU's.

Figure~\ref{fig.tor-ben} shows the non-parallel execution times and
the speed-ups of Algorithm~\ref{alg.newton} using the initial guess
from Algorithm~\ref{alg.inv-curve} with $\eta = 10^{-6}$,
$\eps=10^{-4}$, 135 digits of precision and different number of modes
in the Fourier representation. The figure was done in an Intel Xeon
Gold 5220 CPU at 2.20GHz with 18 CPUs with hyperthreading which
simulates 36 CPUs.

Note that the non-parallel data in Figure~\ref{fig.tor-ben} is
extremely well fit by
\[
T = 0.430018N + 4.75918\ ,
\]
where $T$ denotes the CPU time in seconds and $N$ the Fourier
representation size. This means that the single core time scales
behave (in practice) linearly with the size of the problem. The
logarithmic correction appearing in the theory of the FFT does not
seem to be observable which is due to the fact that the FFT is not the
main problem in the performance of our method. Of course, this is
significantly better than the $N^3$ of Newton's methods based on
inverting matrices.

Algorithm~\ref{alg.newton} also accepts other concurrent computations
such as the FFT algorithm or the solution of the cohomology equations
in steps \ref{alg.newton-cohom1} and \ref{alg.newton-cohom2}. The
latter did not show an important speed up, presumably because the time
spent in these calculations is not so important overall (the number of
points needed is not so large, due to our reduction to 1-D).

\begin{figure}[ht]
 \input{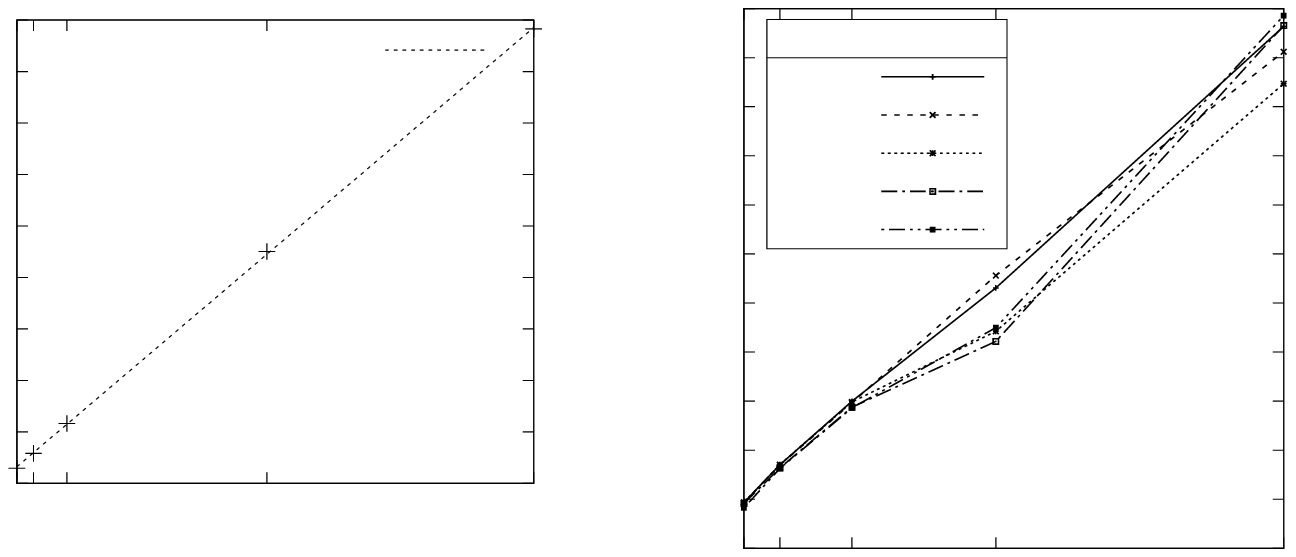}
\hglue-.4cm
\vglue-.4cm
 \caption{On the left: timings not using parallelization and fit of
   the times changing the number of Fourier modes.  On the right:
   Speed-up by a parallel ODE integration in step \ref{alg.newton-PSA}
   of Algorithm~\ref{alg.newton} for different number of threads and
   different number of Fourier modes.}
 \label{fig.tor-ben}
\end{figure}

\subsection{A practical implementation of Algorithm~\ref{alg.newton}}
\label{sec:implementation}

As an example of the implementation of Algorithm~\ref{alg.newton}, we
provide in Figure~\ref{fig.plot} the construction of the invariant
attractors for \equ{eq.spinxy} with a small dissipation, say
$\dis=10^{-6}$, and frequencies given by \equ{eq.omggold} and
\equ{eq.omg}. Figure~\ref{fig.plot2} gives the results for a higher
dissipation, namely $\dis=10^{-3}$.

Both figures show different invariant attractors when the perturbative
parameter $\eps$ changes by using a standard continuation procedure
with interpolation. At each of these continuation steps with respect
to $\eps$ we apply Algorithm~\ref{alg.newton} with a tolerance
$10^{-35}$, which refines the embedding $K_\eps$ and the eccentricity
$\ecc_\eps$, and we also check the accuracy, see
Section~\ref{sec.accuracy-tests}, to ensure that the numerical
solution is accurate enough.

Additionally, we perform an extra refinement at each continuation step
to ensure that the plots in Figures \ref{fig.plot} and \ref{fig.plot2}
make the value of $x$ in the range $[0,2\pi)$. More precisely, if
  $(\tilde K _\eps, \ecc _\eps)$ is the output of
  Algorithm~\ref{alg.newton}, then we assign $K _\eps \gets \tilde K
  _\eps $ with $K _\eps^1(0) = 0$. That is, we apply a shift $\alpha$
  on $\tilde K _\eps$ with $\alpha $ so that the first component of
  $\tilde K _\eps$ at $0$ is zero.

\begin{figure}[ht]
  \input{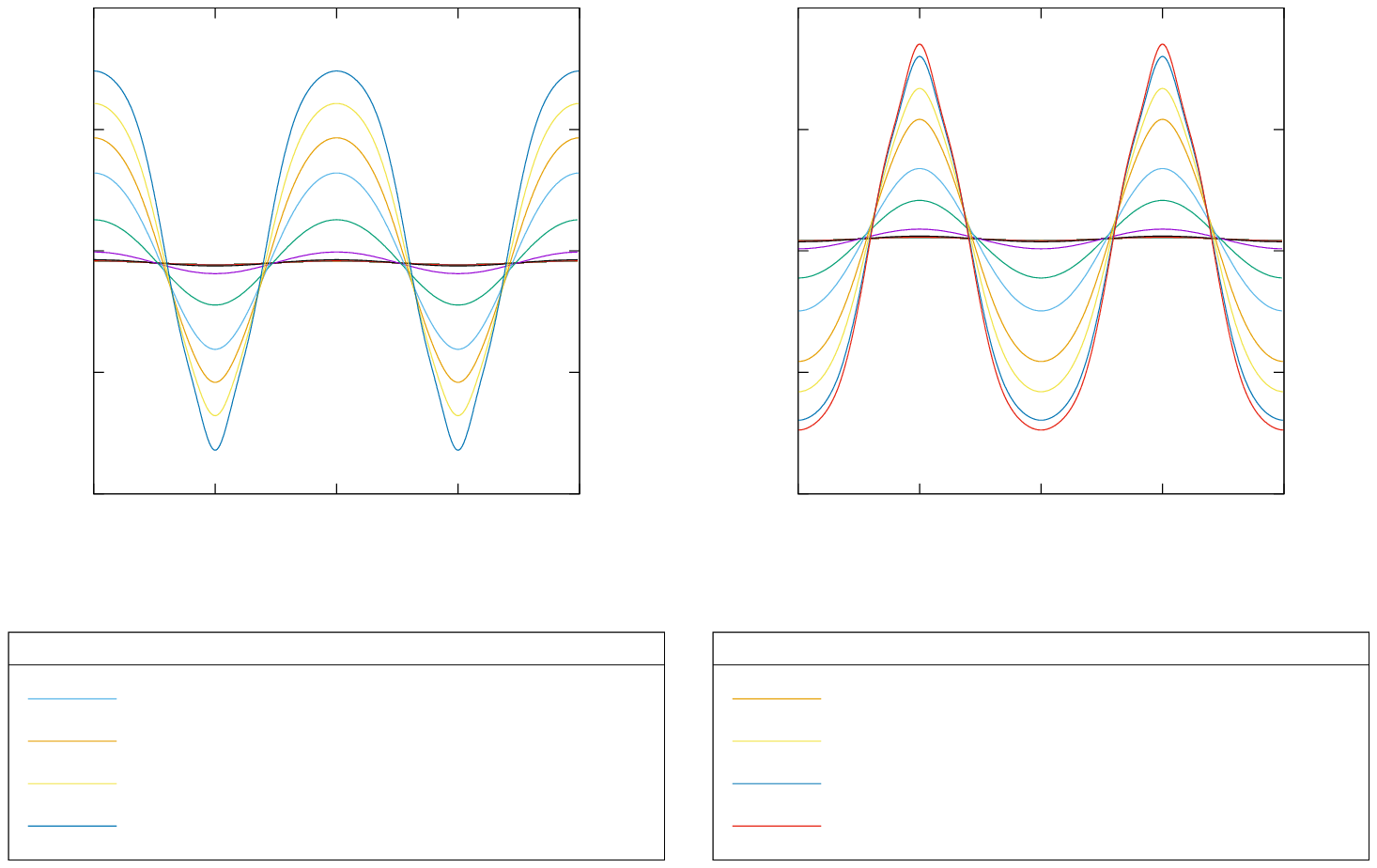}
 \caption{Invariant attractor of \equ{eq.spinxy} with a fixed
   dissipation $\dis=10^{-6}$ after a continuation w.r.t. the
   perturbative parameter $\eps$. On the left, the frequency $\omg$ is
   like in \equ{eq.omg} and on the right as in \equ{eq.omggold}.}
 \label{fig.plot}
\end{figure}

\begin{figure}[ht]
  \input{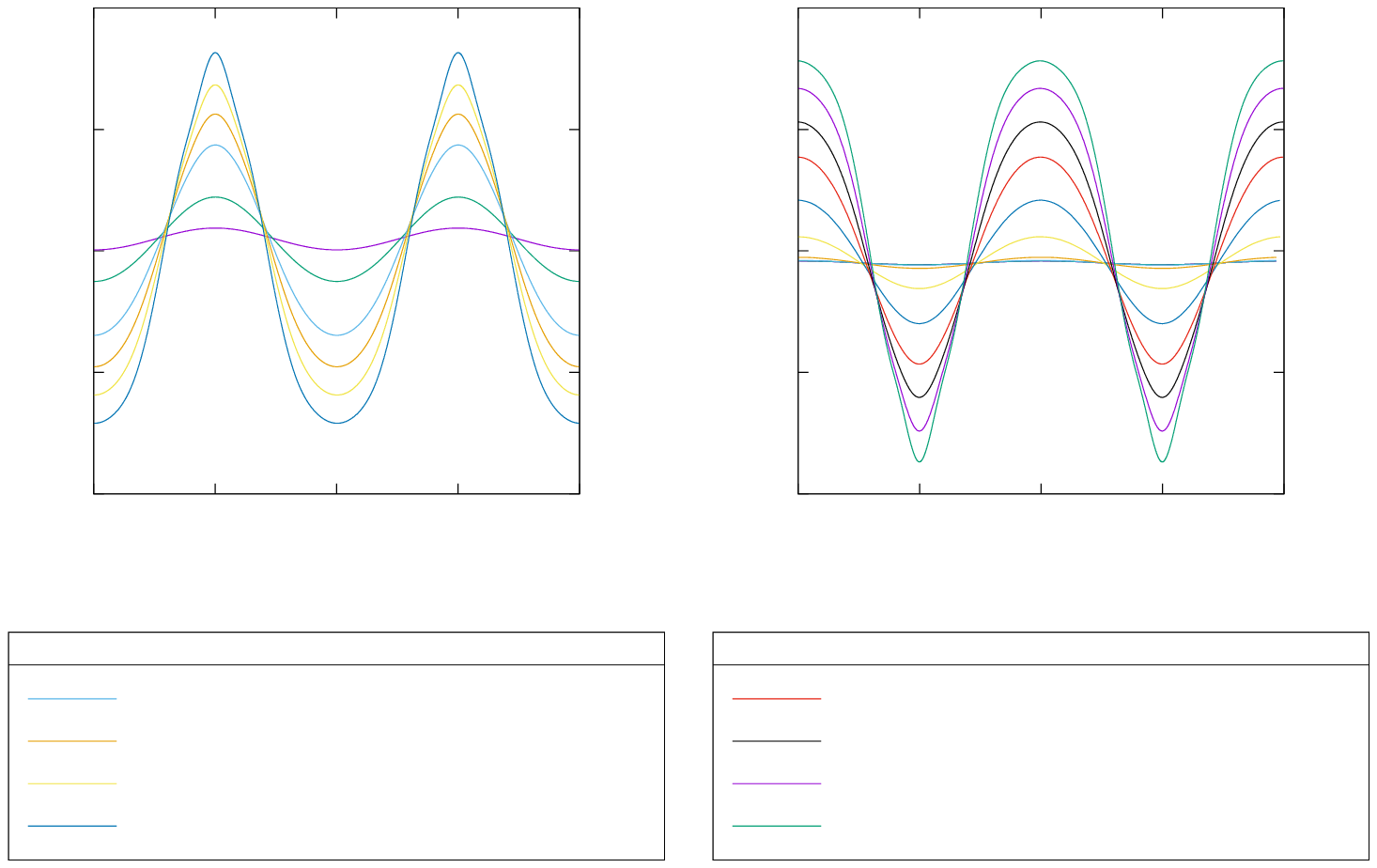}
 \caption{Invariant attractor of \equ{eq.spinxy} with a fixed
   dissipation $\dis=10^{-3}$ after a continuation w.r.t. the
   perturbative parameter $\eps$. On the left, the frequency $\omg$ is
   like in \equ{eq.omg} and on the right as in \equ{eq.omggold}.}
 \label{fig.plot2}
\end{figure}

The results shown in Figure~\ref{fig.plot} give evidence of the
effectiveness of the method of constructing invariant attractors,
using the reduction to a map and the implementation of
Algorithm~\ref{alg.newton}.

\subsection{An empirical comparison between the
  model with time dependent friction \equ{eq.spinxy} and the
  average friction \eqref{eq.avg-spinxy}}

In this section we present a comparison between the numerical results
in the model with time-dependent dissipation \equ{eq.spinxy} and the
model with the averaged dissipation \equ{eq.avg-spinxy}. In
Appendix~\ref{sec:averaging} we present two rigorous justifications of
the averaging procedure.

In Figure~\ref{fig.avnonav}, we plot the difference between the drift
parameters (namely the eccentricities) of the full and averaged models
for the tori with frequencies $\omg _1$ and $\omg _2$; such difference
is small, say of the order of $10^{-7}$, even for parameter values
close to breakdown. We also note that the $y$-axis in
Figure~\ref{fig.avnonav} is just the difference (without absolute
value) which means that $e > e _{avg}$ at each of the $\eps$ values.

\begin{figure}[ht]
  \input{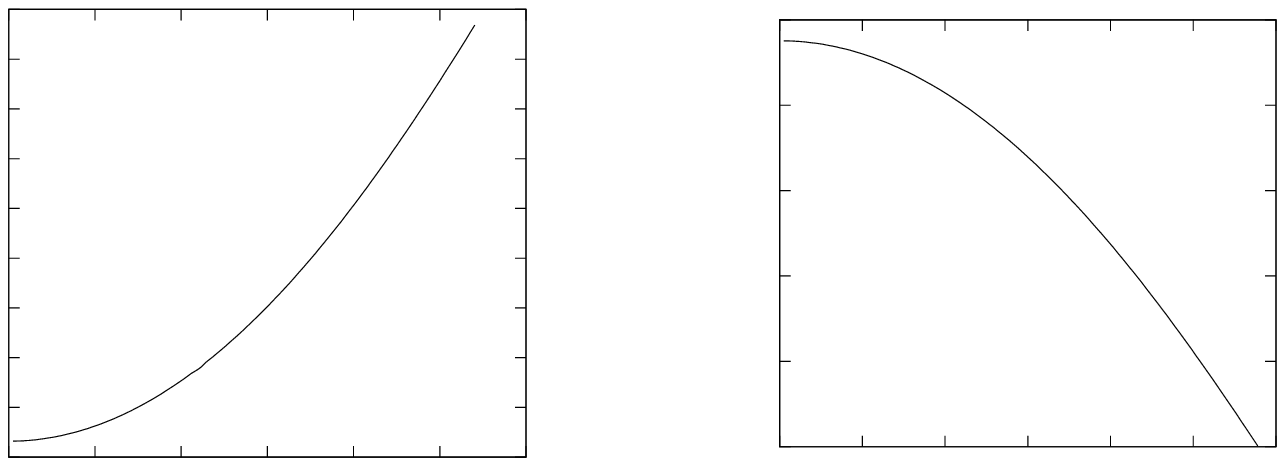}
 \caption{Difference of the eccentricities of the tori using the
   averaged \eqref{eq.avg-spinxy} and the non-averaged
   \eqref{eq.spinxy} spin-orbit model in the case of dissipation
   $\dis=10^{-6}$ and for the two frequencies of interest $\omg _1$ in
   \equ{eq.omggold} on the left and $\omg _2$ in \equ{eq.omg} on the
   right.}
 \label{fig.avnonav}
\end{figure}

\subsection{Comment for double precision accuracy}
The computation with multiprecision is already fast enough thanks to the
quadratic convergence of our method. The use of multiprecision allows
to get tori of arbitrary accuracy and, in particular, it helps to
reach values of the parameter close to the breakdown. These
calculations can be used to identify mathematical patterns close to
breakdown and to use KAM theory.  However, we can also perform double
or even float precision computations for values that are not close to the
breakdown. As one expects, the required time to converge using the
Newton process presented in this paper, see
Algorithm~\ref{alg.newton}, is smaller than in multiprecision, paying
the price of accuracy.

The first reason for this speed is because we request less accuracy
--around $10^{-14}$ for double precision-- which makes the algorithm
converge in fewer iterations.  The second reason is that we avoid
the possible overhead of using a software package, such as
\texttt{MPFR}, and we can then exploit the hardware optimizations
provided by compilers.

After adapting our code to run with double precision we detect that
the values of $\eps$ for which our method still converges is around
$7.8\times 10^{-3}$ which is far from the $\eps \sim 10^{-2}$ using
more accuracy, the speed-up of the parallelization strategy, see
\ref{sec:parallel}, behaves similarly, and the time of a continuation
step (with one CPU) decreases. It still depends on the number of
Fourier modes and some example of values are reported in
Table~\ref{tab.double}.
\begin{table}[ht]
 \[
  \begin{array}{c|*{6}{c}}
    N & 64 & 128 & 256 & 512 & 1024 & 2048 \\ \hline 
    T & 0.570 & 1.132 & 2.280 & 4.618 & 9.329 & 18.049
   \end{array}   
 \]
   \caption{Example of time $T$ (in seconds) of a continuation step
     using double precision, $N$ Fourier modes, no parallelism, $\eta
     = 10^{-3}$, $\omg$ as in \eqref{eq.omggold}, Newton's tolerance
     $10^{-14}$, and integration tolerance $10^{-16}$.
     The data fit extremely well with $T =  0.00884 N + 0.0561$.}
   \label{tab.double}
\end{table}

In certain problems, it may be worth optimizing the speed (at the
price of accuracy and programming time). In such cases where speed is
the most important consideration, it may be worth taking advantage of
modern computer architectures such as GPU's.  (the very structured
nature of our algorithm makes it a tantalizing possibility).  We  
encourage these (or any other) developments on the
methods presented here.


\section{Conclusions}\label{sec:conclusions}

There are several important features of the method developed in this
paper for the construction of invariant tori using the return map; we
highlight below some of these features.

\begin{itemize}
\item The method only requires dealing with functions of a low number
  of variables.
\item The iterative step to construct the tori is quadratically
  convergent.
\item The operation count and the operation requirements are low.
\item The most costly step (integration of the equations) is very easy
  to parallelize efficiently. Many other steps of the algorithm
  involve vector operations, Fourier transforms, which also can take
  advantage of modern computers but here with less impact in the
  performance.
\item The method is well adapted to the rather anisotropic regularity
  of KAM tori and, in the smooth direction, takes advantage of the
  developments in integration of ODE's.
\item The method is reliable, since it is backed by rigorous
  \emph{a-posteriori} theorems.
\item The condition numbers required are evaluated in the approximate
  solution and do not involve global assumptions on the flow such as
  twist.
\item From the theoretical point of view, the a-posteriori theorems
  give a justification of other heuristic methods that produce
  approximate solutions including asymptotic expansions and averaging
  methods.
\end{itemize}

\appendix

\section{Justification of averaging methods}\label{sec:averaging}

In this section, we study the relation between the time dependent
spin-orbit model \eqref{eq.spinxy} and the model \eqref{eq.avg} in
which the dissipation is averaged.

Superficially, the relation between the two models and their
attractors seems to be problematic even in the perturbative regime,
since the existence of quasi-periodic solutions makes assertions for
all the time and conventional averaging methods make assertions only
for times which are inverse powers of the perturbation.

We will present two arguments showing that, indeed for the problem of
the existence of quasi-periodic tori, the averaging method produces an
accurate result.  The first argument (Section~\ref{sec:averaging1}
will apply to very general models, but will produce results for KAM
tori. The second method (Section~\ref{sec:averaging2} will be very
specific for the spin-orbit problem but will provide information for
all orbits.

We note, however, that our method produces not only a rigorous
estimate on the error of the averaging method, but also suggests other
approximations that are more accurate than the usual procedure of
averaging only the dissipation (see Section~\ref{sec:averaging2}).
Our results show that, if besides averaging the dissipation, one
changes slightly the conservative forces (we give explicit formulas),
one gets results that are more accurate than just averaging the
friction.

In this section we will consider a differential equation of the form
\begin{equation}\label{generalform}
  \frac{d^2 x(t)}{dt^2}  +  a(t) \frac{d x(t)}{dt} + F(x(t),t) =
  0\ .
\end{equation}

In the perturbative case, we will have that there is a small parameter
in the time dependence $a(t) = \mu \alpha(t)$.

\subsection{Perturbative arguments based on a-posteriori theorems}
\label{sec:averaging1}

We first observe that the key to our result is the study of the
time-one map of the vector field.

If the return maps of two vector fields are close (in a smooth enough
norm) and they satisfy the non-degeneracy conditions, the solutions of
the invariance equation \eqref{invariance} for the two systems are
close.

To detail this remark, let us consider two Poincar\'e maps $P_e^{(A)}$
and $P_e^{(C)}$, depending on a drift term $e$ and satisfying some
non-degeneracy conditions; for example, we can take $P_e^{(A)}$ as the
Poincar\'e map of the averaged system and $P_e^{(C)}$ as the
Poincar\'e map of the complete, non-averaged system.

Let $(K_A,e_A)$ be an approximate solution of the invariance equation
as in \equ{approx} with a small error term $E_A$:
\[
P_{e_A}^{(A)}\circ K_A(\theta)-K_A(\theta+\omg)=E_A(\theta)\ .
\]
Assume that the maps $P_e^{(C)}$, $P_e^{(A)}$ are close (in a
suitable norm); then $(K_A,e_A)$ is an approximate solution for
$P_e^{(C)}$; in fact, we have that
\beqano
P_{e_A}^{(C)}\circ
K_A(\theta)-K_A(\theta+\omg)&=&(P_{e_A}^{(C)}-P_{e_A}^{(A)})\circ
K_A(\theta)+P_{e_A}^{(A)}\circ
K_A(\theta)-K_A(\theta+\omg)\nonumber\\ &=&(P_{e_A}^{(C)}-P_{e_A}^{(A)})\circ
K_A(\theta)+E_A(\theta)\equiv E_C(\theta)
\eeqano
with $E_C$ small. Using the KAM theory given in
\cite{CallejaCdlL2013}, there exists $(K_C,e_C)$ such that
\[
P_{e_C}^{(C)}\circ K_C(\theta)-K_C(\theta+\omg)=0\ ;
\]
besides, the results in \cite{CallejaCdlL2013} show that $K_C$ is
close to $K_A$ and $e_C$ is close to $e_A$.

The importance of this remark is that to obtain the distance of the
maps, we only need to estimate the difference between the evolution
for a time $2 \pi$. In the perturbative regime, when the oscillation
is of size $\mu$ controlling the distance between the solution of the
averaged and non-averaged systems is well within the reach of standard
averaging methods \cite{Hale, SandersV2007} (which deal well with
times $O(\mu^{-1})$, much larger than $2 \pi$).


\subsection{Non-perturbative arguments based on elementary
scaling of variables and of time} \label{sec:averaging2}

In this section, we present a completely elementary (non perturbative)
technique to justify the averaging method.

We just observe that if $x$ satisfes \eqref{generalform} then, for any
non-zero (smooth) function $\gamma(t)$, the function $y(t)$ defined by
$x(t) = \gamma(t)y(t)$ satisfies:
\begin{equation}\label{yequation}
  \gamma(t)   \frac{d^2 y(t)}{dt^2}
  +( 2 \gamma'(t)  + a(t) \gamma(t) )  \frac{d y(t)}{dt}
  + (\gamma''(t) + a(t) \gamma'(t) ) y(t)  + F(\gamma(t) y(t),t) =
  0\ .
\end{equation}
Let $\bar{a}$ be the average value of $a(t)$; if we choose
\begin{equation}\label{goodgamma}
\gamma(t) = \exp \biggl( -\frac{1}{2} \int_0^t ( a(s) - \bar{a}) \, ds
\biggr)\ ,
\end{equation}
which is a periodic function, the equation~\eqref{yequation} becomes:
\begin{equation}\label{averaged}
  \frac{d^2 y(t)}{dt^2}
  +\bar{a}  \frac{d y(t)}{dt} + G(y(t),t)  = 0\ ,
\end{equation}
where
\begin{equation} \label{Gfunction}
 G(y,t) \coloneq
 \frac{ \gamma''(t) + a(t) \gamma'(t)}{\gamma(t)}
  y  + {1\over {\gamma(t)}} F(\gamma(t) y,t)\ .
\end{equation}

Hence, the function $y$ satisfies the equation with an average
dissipation (and a different $F$).  Notice that $\gamma(0) = \gamma(2
\pi) = 1$, so that
\begin{equation} \label{change}
  \begin{split}
    x (0) & = y(0) \gamma(0)\\
    x(2 \pi) & = y(2 \pi) \gamma(2\pi)\\
y'(0) & = x'(0) + \frac{1}{2}(a(0) - \bar{a}) x(0)\\
y'(2\pi) & = x'(2\pi) + \frac{1}{2}(a(2\pi) - \bar{a}) x(2\pi)\ .
\end{split}
\end{equation}

We can think of \eqref{change} as a change of variables in phase space
from $(x,x')$ to $(y, y')$.  The return map for the averaged equation
\eqref{averaged} in the variables $(y,y')$ is equivalent to the
orginal problem.

Hence, we can read off the original return map as the return map of
the averaged equation under a change of variables.

Note that, as standard in the averaging method, the relation between
the averaged equations and the true ones is mainly a change of
variables and a modification to the equations. Note, that in our very
simple equations, the equivalence is exact.

In the perturbative case, we obtain that $ \gamma = 1 + O(\eps)$ and
that the function $G$ in \eqref{Gfunction} satisfies $G - F =
O(\eps)$; also the change of variables between $(x,x')$ and $(y, y')$
is $O(\eps)$ close to the identity.

Our treatment shows that models with average dissipation should
involve also a change of variables and a modification of the forces.

\section{Taylor's integration methods}
\label{sec.taylorint}
Taylor's method is one of the most common numerical integration
techniques of an initial value problem of an ordinary differential
equation of the form
\begin{equation}
\label{eq.ivp}
 \begin{split}
   \dot z &= F(z,t), \\
   z(t _0) &= z _0\ .
 \end{split}
\end{equation}
The Taylor's method is competitive, in speed and accuracy, with
respect to other standards methods. The main drawback is that the
Taylor's method is an explicit method, so it has all limitations of
these kind of schemes, for example it is non-appropriate for stiff
systems.

The idea behind the Taylor's method is very simple. Given the initial
condition $z (t _0) = z _0$, the value $z (t _1)$, with $t _1 = t _0 +
h$, is approximated by the Taylor series of $z(t)$ at $t = t _0$. The
Taylor series is truncated up to an order, say $N$, to try to ensure
the absolute/relative tolerances requested during the numerical
integration. Therefore, to get the solution $z_1$ at time $t=t_1$ from
the solution $z_0$ at time $t=t_0$ we consider the expression
\begin{equation}
\label{eq.z1}
 z_1=z_1+z_0^{[1]}h+z_0^{[2]}h^2+ \dotsb +z_0^{[N]}h^N\ ,
\end{equation}
where $z_0^{[k]}$, $k=1, \dotsc ,N$, represents the normalized
derivative at order $k$ computed at $t_0$, i.e.,
\[
z_0^{[k]}={1\over {k!}}\ {{d^kz}\over {dt^k}}(t_0)\ .
\]
Using the coefficients of the Taylor's expansion in \eqref{eq.z1} one
can estimate the range of $h$ where the Taylor series is valid (up to
a tolerance). This fact makes the Taylor's method suitable for
multi-precision arithmetic.

The computation of the derivatives might be a difficult task, which
can be lightened by using automatic differentiation (see, for
instance, \cite{RallC1996,GriewankW2008}), thus providing very
efficient implementations of Taylor's method as illustrated in
\cite{JZ05} to which we refer for full details. We recall that
automatic differentiation provides a recursive computation of
operations on polynomials, which implies the manipulation of formal
power series.

We also mention that jet transport (see Appendix~\ref{sec:jet}),
namely automatic differentiation with respect to initial data and
parameters, can be used in Taylor's method to approximate the high
order variational flow as it has been proved in \cite{Gimeno2021}.

\section{Variational equations for the spin-orbit problem}
\label{sec:variational}

In this section we provide the formulae for the computation of the
variational equations with respect to coordinates and parameters,
motivated by the fact that - even if not used in the present paper -
they are useful in different contexts, like the parameterization of
invariant objects, estimates based on derivatives of the flow,
computation of chaos indicators.

 \subsection{Variational equations}
 \label{sec.vars}
The variation with respect to the initial conditions of
\eqref{eq.spinbg} involves the Jacobian whose elements are given by
\begin{equation}
 \label{eq.bjacob}
 \begin{aligned}
  a _{11}&= 0, & a _{21}&= -2\eps \frac{a}{r(u;\ecc)}c(\beta;u,\ecc), \\
  a _{12}&= 1, & a _{22}&= \ecc \frac{a}{r(u;\ecc)} \sin u - \dis
  \left(\frac{a}{r(u;\ecc)}\right)^5
 \end{aligned}
\end{equation}
with $c(\beta;u,\ecc)$ defined in \eqref{eq.cx}.

The variations with respect to the initial conditions for the system
\eqref{eq.spinbg} or for the system \eqref{eq.spinxy} are not the
same, although some properties such as the determinant or the
eigenvalues of the $2\pi$-time map are preserved. The explicit
relation between the two variations is given by
\begin{equation}
 \label{eq.varb2x}
 \begin{aligned}
  \frac{\partial x}{\partial x _0}(t) &= \frac{\partial
    \beta}{\partial \beta _0}(u), & \frac{\partial x}{\partial y
    _0}(t) &= \frac{\partial \beta}{\partial \gamma _0}(u) (1 - \ecc
  \cos u _0) , \\
  \frac{\partial y}{\partial x _0}(t) &= \frac{\partial
    \gamma}{\partial \beta _0}(u) \frac{1}{1 - \ecc \cos u}, &
  \frac{\partial y}{\partial y _0}(t) &= \frac{\partial
    \gamma}{\partial \gamma _0}(u) \frac{1 -\ecc \cos u _0}{1 - \ecc
    \cos u}\ ,
 \end{aligned}
\end{equation}
where $(x _0, y _0 ) = (\beta _0, \gamma _0/(1 - \ecc \cos u
_0))$. The relation \eqref{eq.varb2x} must be interpreted as follows:
after the integration of \eqref{eq.spinbg} and its first variational
equations, which uses the terms in \eqref{eq.bjacob}, from time $u _0$
to $u$ and with initial condition $(\beta _0, \gamma _0) \in [0, 2\pi)
  \times \mathbb{R}$, then the variation with respect to the initial
  condition $(x _0, y _0)$ in \eqref{eq.spinxy} from the initial time
  $t _0 = u _0-\ecc \sin u _0$ to the final time $t = u - \ecc \sin u$
  is given by the relations in \eqref{eq.varb2x}.\\ Note that
  \eqref{eq.varb2x} is simplified when $u _0 = 0$ and $u = 2\pi$.

\subsection{Variational equations with respect to the parameters}
The variational equations with respect to the parameters $\eps$ and
$\dis$ of the system \eqref{eq.spinbg} are quite straightforward as
well as their relations in terms of the variables $(x,y)$. They are
given by the following expressions:
\begin{equation}
\label{eq.varparamb2x}
 \frac{\partial x}{\partial \bigstar} (t) = \frac{\partial
   \beta}{\partial \bigstar}(u) \text{ and } \frac{\partial
   y}{\partial \bigstar} (t) = \frac{\partial \gamma}{\partial
   \bigstar}(u) \frac{a}{r}, \qquad \bigstar \in \{\eps, \dis\}\ .
\end{equation}
However, the case for the parameter $\ecc$ in \eqref{eq.spinbg}
requires a little bit more of work and its relation with respect to
the coordinates $(x,y)$ also involves more terms which we make
explicit below:
\begin{equation}
 \label{eq.vareccb2x}
 \begin{split}
  \frac{\partial x}{\partial \ecc}(t) &=
  \begin{aligned}[t]
  &\frac{\partial \beta}{\partial \ecc}(u) - \frac{\partial \beta
    }{\partial \gamma _0}(u) \frac{\gamma _0 \cos u _0}{1 - \ecc \cos
      u _0} + \gamma(u)\frac{a}{r} \sin u\\ & + \bigg[\frac{\partial
        \beta}{\partial u _0}(u) + \frac{\partial \beta}{\partial
        \gamma _0}(u) \frac{\gamma _0 \ecc \sin u _0}{1 - \ecc \cos u
        _0}\biggr]\frac{\sin u _0}{1 - \ecc \cos u _0}\ ,
  \end{aligned} \\
  \frac{\partial y}{\partial \ecc}(t) &=
  \begin{aligned}[t]
   &\frac{\partial \gamma}{\partial \ecc}(u) \frac{a}{r} +
    \gamma(u)(\frac{a}{r})^2\cos u - \frac{\partial \gamma}{\partial
      \gamma _0}(u) \frac{a}{r} \frac{\gamma _0 \cos u _0}{1 - \ecc
      \cos u _0}\\ & + \biggl[ \frac{\partial \gamma }{\partial u}(u)
      - \gamma (u) \frac{a}{r}\ecc \sin u \biggr](\frac{a}{r})^2 \sin
    u \\ & +\biggl[ \frac{\partial \gamma }{\partial u _0}(u) +
      \frac{\partial \gamma}{\partial \gamma _0}(u) \frac{a}{r}
      \frac{\gamma _0 \ecc \sin u _0}{1 - \ecc \cos u _0} \biggr]
    \frac{\sin u _0}{1 - \ecc \cos u _0}\ .
  \end{aligned}
 \end{split}
\end{equation}

Note that \eqref{eq.vareccb2x} is simplified when $u _0 = 0$ and $u =
2\pi$.

\subsection{High order variational equations}
In many cases the first order variational equations are
straightforward and one can explicitly write them down in the
numerical integrator.  However, high order variational equations are
cumbersome and the use of jet transport becomes highly recommended,
see \cite{Gimeno2021}.  The jet transport is also useful in the study
of other structures such as stable manifolds (which will not be
considered here).

Jet transport uses automatic differentiation (see
\cite{GriewankW2008}), which manipulates multivariate polynomials to
carry out the truncated Taylor's approximation containing (in the case
of jet transport) the higher order variational flow, see
\cite{Gimeno2021} for a precise formulation. We also refer to
Appendix~\ref{sec.taylorint} for a discussion of a polynomial
manipulator up to degree $2$.

In the case of the spin-orbit problem given by \eqref{eq.spinbg}, the
polynomial manipulator must at least contain the sum, product, sine,
cosine, and power operations.  All of them have explicit recurrence
expressions (\cite{Knuth97,Haro2016}).  By the use of the polynomial
manipulator, expressions such as \eqref{eq.vareccb2x} for higher
orders are automatically obtained.

\section{Different models of computer arithmetic}
\label{sec:multiple}

The Algorithm~\ref{alg.newton} can be implemented with multi precision
arithmetic.  The idea is that there are different models of computer
arithmetic implementation. It is very easy to switch between different
models using features of modern languages such as overloading.

We recall that computers deal only with representable numbers, which
are just a (finite) subset of the real numbers and whose elements have
the form
\begin{equation}
 \label{eq.floatingnum}
 \pm m \cdot \beta ^{e-t}\ ,
\end{equation}
where $m \in [0, \beta ^{t}-1]$ is an integer called mantissa, $\beta$
is the base or radix (typically $\beta = 2$), $t$ is a positive
integer denoting the precision and $s\in [s_{min}, s_{max}]$ is also
an integer called the exponent.  Typical values in the double
precision arithmetic following the IEEE 754 standard\footnote{We omit
  the discussion of \emph{``denormalized numbers''}, perhaps the
  aspect of IEEE 754 that generated the most controversy.} are $\beta
= 2$, $t=53$, $s _{min} = -1021$, and $s _{max} = 1024$.

Performing arithmetic operations on representable numbers, very often
yields numbers that are not representable. Sometimes, the results are
in the middle of representable numbers and then, one assigns the
result to one of the neighboring numbers.  This is called
\emph{rounding}. There are several rules that are in common use:
rounding to nearest, rounding down, rounding up, rounding towards
zero, rounding away to zero. There are operations that yield a number
that is far from any representable number (e.g., adding the largest
representable number to itself). One usually represents those as
\emph{Inf}. Of course, one needs to have rules on how to deal with
\emph{Inf} (e.g., does one distinguish between positive/negative
infinity).  Finally, there operations that do not make sense, such as
division by zero.

In the IEEE 754 standard, it is specified that there is a so called
\emph{control word}. The bits of the control word specify the rounding
modes and how to treat infinity.

Finally, we just note that, when the result of two numbers is very
close to zero, rounding leads to a great loss of precision. IEEE 754
has introduced also the \emph{denormalized numbers} which make the
gaps near zero smaller, even if they require specialized rules to be
handled.

We also note that the IEEE 754 standard specifies the calculations of
some transcendental functions subject to the same rules of rounding.

Nowadays, the IEEE 754 standard is implemented in hardware. Both in
CPU's and in many GPU's. Many important libraries (BLAS/ATLAS, FFTW3)
take advantage of the availability of hardware and obtain advantages
in speed.  Having very reliable rounding modes that satisfy identities
allows to improve also the accuracy. For example, there are algorithms
that sum a sequence of numbers with a roundoff error independent of
the length of the sequence to be summed. One can also use
\emph{interval arithmetic} that provides rigorous estimates of the
arithmetic operations.

Nowadays, there exist libraries that provide the same capabilities
indicated above (representable numbers of the standard form, rounding
modes, etc.), but with a number of digits that can be selected at run
time; for example, {\tt MPFR} \cite{FousseHLPZ2007} (which is the one
we have used) or some other.

Using modern programming techniques such as overloading, it is not
difficult to write versions of our programs for the arithmetic using
the hardware or in {\tt MPFR} and select the precision.

It is important to remark that in a parallel scenario, one must be
sure that the global variables used in these libraries are initialized
in each of the different threads. In particular, in the {\tt MPFR}
case, one needs to initialize the precision and the rounding for each
of the threads, otherwise the output will differ from the non-parallel
version.

\section{A goodness test for jet transport}\label{sec:jet}
Let $\dot x = f(t,x)$ an ODE in $\mathbb{R}^d$ with flow denoted by
$\varphi(t; x)$. The test consists in running three integrations with
two different integrators.

First, we run $x _0 + s$ with an integrator with jet transport of
order, say $N$, up to a time, say $1$. The output is the jet $y(s)$ of
order $N$.

Second, we choose a unitary vector $v$ and a scalar $h$ (typically
small, say $h=10^{-7}$) and we define the quantity
\begin{equation*}
 c _h = \| y(h) - \varphi(1; x _0 + hv) \|\ ,
\end{equation*}
which is expected to behave as $c _h \approx c h ^{N+1}$ for some $c >
0$.

The third and last run is to repeat the second one, but now with $h/2$
to get the quantity $c _{h/2}$.

Finally, it must happen that
\begin{equation}
 \label{eq.ninjaquo}
 \frac{c _h}{c _{h/2}} \approx 2 ^{N+1}.
\end{equation}
Therefore, the test is successful when $\log _2(c _{h}/c _{h/2})
\approx N+1$, being $N$ the order of jets in the first integration.

Notice that \eqref{eq.ninjaquo} may suffer loss of precision if $h$ is
too small. Therefore, one needs to choose a suitable $h$ by
systematically trying several choices.


\newcommand{\etalchar}[1]{$^{#1}$}
\def\cprime{$'$} \def\cprime{$'$} \def\cprime{$'$} \def\cprime{$'$}

\end{document}